\documentclass[12pt,final]{article}   

\usepackage{amsmath}
\usepackage{amsthm}
\usepackage{amssymb}
\usepackage{amsfonts}
\usepackage{latexsym}               
\usepackage{amsgen}
\usepackage[mathscr]{eucal}
\usepackage{mathrsfs}
\usepackage{bm}
\usepackage{enumerate}
\usepackage{cite}

\usepackage{mathptmx} 
\usepackage{color}
\usepackage[dvips]{graphicx}
\usepackage{mediabb}

\usepackage{showkeys}

\newtheorem{lem}{Lemma}[section]
\newtheorem{thm}[lem]{Theorem}
\newtheorem{pro}[lem]{Proposition}
\newtheorem{cor}[lem]{Corollary}
\newtheorem{defn}[lem]{Definition}

\newcommand{\bqn}{\begin{equation}}
\newcommand{\eqn}{\end{equation}}
\newcommand{\beqx}{\begin{equation*}}
\newcommand{\eeqx}{\end{equation*}}
\newcommand{\barr}{\begin{array}}
\newcommand{\earr}{\end{array}}
\newcommand{\beqn}{\begin{eqnarray}}
\newcommand{\eeqn}{\end{eqnarray}}
\newcommand{\beqnx}{\begin{eqnarray*}}
\newcommand{\eeqnx}{\end{eqnarray*}}
\newcommand{\bmt}{\begin{multline}}
\newcommand{\emt}{\end{multline}}

\numberwithin{equation}{section}

\newcommand{\D}{\partial}

\newcommand{\supp}{\operatorname{supp}}

\newcommand{\ve}{\varepsilon}

\newcommand{\vphi}{{\varphi}}

\newcommand{\ro}{\rho}

\newcommand{\er}{\eqref}
\newcommand{\lb}{\label}
\newcommand{\qu}{\quad}

\newcommand{{\R}}{\mbox{$\mathbb R$}}

\begin{document}


\title{The Kinetic and Hydrodynamic Bohm Criterions \\
for Plasma Sheath Formation}

\author{Masahiro Suzuki${}^1$ and Masahiro Takayama${}^2$}

\date{%
\normalsize
${}^1$%
Department of Computer Science and Engineering, 
Nagoya Institute of Technology,
\\
Gokiso-cho, Showa-ku, Nagoya, 466-8555, Japan
\\ [7pt]
${}^2$%
Department of Mathematics, 
Keio University, 
\\
Hiyoshi, Kohoku-ku, Yokohama, 223-8522, Japan
\\[10pt]
\large \today
}

\maketitle

\begin{abstract}
The purpose of this paper is to mathematically investigate the formation of 
a plasma sheath, and to analyze the Bohm criterions which are required for the formation.
Bohm derived originally the (hydrodynamic) Bohm criterion from the Euler--Poisson system.
Boyd and Thompson proposed the (kinetic) Bohm criterion from a kinetic point of view,
and then Riemann derived it from the Vlasov--Poisson system.
In this paper, we prove the solvability of boundary value problems of the Vlasov--Poisson system.
On the process, we see that  the kinetic Bohm criterion is a necessary condition for the solvability.
The argument gives a simpler derivation of the criterion. 
Furthermore, the hydrodynamic criterion can be derived from the kinetic criterion.
It is of great interest to find the relation between 
the solutions of the Vlasov--Poisson and Euler--Poisson systems.
To clarify the relation, we also study the delta mass limit of solutions of the Vlasov--Poisson system.

\end{abstract}

\begin{description}
\item[{\it Keywords:}]
Vlasov--Poisson system,
Euler--Poisson system,
Boundary value problem,
Delta mass limit.

\item[{\it 2020 Mathematics Subject Classification:}]
35A01, 
35M32, 
35Q35, 
76X05 

\end{description}


\newpage

\section{Introduction}\lb{S1}

The purpose of this paper is to mathematically investigate the formation of 
a plasma boundary layer, called as a \emph{sheath}, 
near the surface of materials immersed in a plasma, 
and to analyze the Bohm criterions which are required for the formation.
The sheath appears when a material is surrounded
by a plasma and the plasma contacts with its surface.
Because the thermal velocities of electrons are much higher than those of ions, 
more electrons tend to hit the material compared with ions.
This makes the material negatively charged with respect to the surrounding plasma. 
Then the material with a negative potential attracts and accelerates ions toward the surface, 
while repelling electrons away from it. 
Eventually, there appears a non-neutral potential region near the surface, 
where a nontrivial equilibrium of the densities is achieved. 
This non-neutral region is referred as to the sheath. 
For more details of  physicality of the sheath development, 
we refer the reader to \cite{D.B.1, F.C.1,I.L.1, LL, K.R.1, K.R.2}.

For the formation of sheath, Langmuir \cite{I.L.1} observed that 
positive ions must enter the sheath region with a sufficiently large flow velocity. 
Using the Euler--Poisson system (see \eqref{sp0} below),
Bohm \cite{D.B.1} proposed the original \emph{Bohm criterion}
which states that the flow velocity of positive ions at the plasma edge must exceed the ion acoustic speed.
In this paper, we call the criterion \emph{a hydrodynamic Bohm criterion}. 
Nowadays there are many mathematical results
which investigated the sheath formation by using the Euler--Poisson system.
The studies \cite{A.A.,AMR,NOS,M.S.1,M.S.2,MM1}  established
the existence and stability of stationary solutions assuming the hydrodynamic Bohm criterion.
These results validated mathematically the criterion.
From different perspectives to the those results',
the sheath formation was discussed by considering the quasi-neutral limit 
as letting the Debye length in the Euler--Poisson system tend to zero \cite{GHR1,GHR2,JKS1,JKS2,JKS3}.
Furthermore, Feldman--Ha--Slemrod \cite{FHS1} studied the sheath formation adopting 
a certain hydrodynamic model which describes the dynamics of an interface between the plasma and sheath
(see also \cite{RD1}).

From a kinetic point of view, Boyd--Thompson \cite{BT1} proposed a Bohm criterion.
After that Riemann \cite{K.R.1} derived it from the Vlasov--Poisson system (see \eqref{sVP1} below).
We call it \emph{a kinetic Bohm criterion} in this paper. 
There is no mathematical result which investigates the sheath formations 
by using the Vlasov--Poisson system with rigorous proofs.
One of our goals is to prove the solvability of boundary value problems of 
the Vlasov--Poisson system under the kinetic Bohm criterion.
On the process, we see that the criterion is a necessary condition for the solvability.
The argument gives a simpler derivation of the criterion. 

It is worth pointing out that the hydrodynamic Bohm criterion can be derived from the kinetic Bohm criterion.
We are also interested in finding the relation between the solutions of
the Vlasov--Poisson and Euler--Poisson systems.
To clarify the relation, we study some limit of solutions of the Vlasov--Poisson system.

After a suitable nondimensionalization, 
the stationary Vlasov--Poisson system is written by
\begin{subequations}\label{sVP1}
\begin{gather}
\xi_{1} \partial_{x} f+\partial_{x} \phi \partial_{\xi_{1}} f =0, \quad x>0 , \ \xi \in \mathbb R^{3},
\label{seq1}
\\
\partial_{xx} \phi 
= \int_{\mathbb R^{3}} f d\xi - n_{e}(\phi), \quad x>0,
\label{seq2}
\end{gather}
where $x>0$ and $\xi = (\xi_{1},\xi_{2},\xi_{3})=(\xi_{1},\xi') \in \mathbb R^{3}$ 
are the space variable and velocity, respectively.
The unknown functions $f = f(x,\xi)$ and $-\phi=-\phi(x)$ stand 
for the velocity distribution of positive ions
and the electrostatic potential, respectively. 
The given function $n_{e}(\cdot) \in C^{2}(\mathbb R)$ denotes the number density of electrons.
We assume that 
\begin{gather*}
n_{e}(0)=1, \quad n_{e}'(0)=-1.
\end{gather*}
One of typical functions is the Boltzmann relation $n_{e}(\phi)=e^{-{\phi}}$.

We study the boundary value problem of \eqref{seq1}--\eqref{seq2} with the boundary conditions
\begin{gather}
f (0,\xi) = f_{b}(\xi)+\alpha  f(0,-\xi_{1},\xi'), \quad \xi_{1}>0,
\label{sbc1} \\
\phi(0)=\phi_{b},
\label{sbc3} \\
\lim_{x \to\infty} f (x,\xi) =  f_{\infty}(\xi), 
\quad \xi \in \mathbb R^{3},
\label{sbc2} \\
\lim_{x \to\infty} \phi (x) =  0.
\label{sbc4} 
\end{gather}
\end{subequations}
The constants $\alpha \in [0,1]$ and $-\phi_{b} \in \mathbb R$
denote the rate of refraction and the voltage on the boundary, respectively.
Furthermore, $f_{b}=f_{b}(\xi)$ and $f_{\infty}=f_{\infty}(\xi)$ are given nonnegative functions.
Physically speaking, the case $f_{b}=0$ and $\alpha=0$ corresponds to {\it a completely absorbing wall}.
In addition, $\phi_{b}>0$ and $\phi_{b}<0$ mean
that the wall is negatively and positively charged, respectively.
Let us say {\it attractive} and {\it repulsive} boundaries for positive ions 
if $\phi_{b}>0$ and $\phi_{b}<0$, respectively.

Riemann \cite{K.R.1} studied essentially the same boundary value problem as \eqref{sVP1} to derive the kinetic Bohm criterion
\begin{gather}\label{Bohm1}
\int_{\mathbb R^{3}}\xi_{1}^{-2}f_{\infty}(\xi) d\xi  \leq 1.
\end{gather}
He also assumed that the number density of electrons is given 
by a function of the electrostatic potential as in \eqref{seq2}.
On the other hand, he did not impose any boundary condition for the potential $\phi$ at $x=0$,
and supposed implicitly a situation that $\phi$ is monotone. 
The condition \eqref{sbc3} is one of simplest boundary conditions that can create the situation.
In stead of \eqref{sbc3}, we can also impose the following boundary condition:
\begin{gather}\label{sbc5}
n_{e}(\phi(0)) v_{e}=(1-\alpha)\int_{\mathbb R^{3}} \xi_{1} f(0,\xi) d\xi,
\end{gather}
where $v_{e}$ is a constant. 
This condition means physically that the outgoing fluxes of electrons and ions coincide at the boundary.
Even if we consider a boundary value problem of 
\eqref{seq1}--\eqref{seq2} with conditions \eqref{sbc1},  \eqref{sbc2}, \eqref{sbc4}, and \eqref{sbc5}, 
it is seen that a value $\phi(0)$ is determined a priori, 
and thus we can reduce the problem to the boundary value problem \eqref{sVP1} (for more details, see Appendix \ref{A0}).
Therefore, we focus ourself on the study of the problem \eqref{sVP1} in this paper.

The derivation of \eqref{Bohm1} by Riemann \cite{K.R.1} is clear to understand, but it seems to be simplified.
Indeed he divided the two cases $\xi_{1}>0$ and $\xi_{1}<0$, 
and then changed the coordinates according to $\xi_{1} \gtrless 0$.
Some expansion of the unknown function $\phi$ was also used.
One of our purposes is to find a simpler derivation 
avoiding the use of such coordinate transformations and expansions.


The Vlasov--Poisson system in \eqref{sVP1}
is given by a system of partial and ordinary differential equations, 
while the stationary Euler--Poisson system is given 
by a system of just ordinary differential equations.
Besides $\eqref{seq2}$ has a non-local term.
The idea to resolve these difficulties is to reduce the problem \eqref{sVP1} to 
a boundary value problem of a first-order ordinary differential equation only for $\phi$
by combining the characteristics method and the technique used in \cite{M.S.1}. 
The reduction also enables us to derive more simply the kinetic Bohm criterion \eqref{Bohm1},
which is a necessary condition of the solvability of \eqref{sVP1}.
We will construct the solution under \eqref{Bohm1}
(see Theorems \ref{existence1}, \ref{existence2}, and \ref{existence3} below).

It is worth pointing out that letting $f_{\infty}(\xi)$ be a delta function $\delta(\xi-(-u_{\infty},0,0))$ 
in the kinetic Bohm criterion \eqref{Bohm1}, one can obtain
the hydrodynamic Bohm criterion:
\begin{gather}\label{Bohm4}
u_{\infty}^{2} \geq 1,
\end{gather}
where $u_{\infty}$ is a positive constant and
$-u_{\infty}$ means the flow velocity of positive ions at infinite distance.
Bohm derived originally the criterion \eqref{Bohm4} by using 
the stationary Euler--Poisson system of cold plasma:
\begin{subequations}\lb{sp0}
\begin{gather}
   (\rho u)'=0, \quad 
   {u}u'={\phi}', \quad 
   \phi''=\rho-e^{-{\phi}}, \quad x>0,
  \lb{sp3}
   \end{gather}
where $\rho=\rho(x)$, $u=u(x)$, and $-\phi=-\phi(x)$
represent the number density and flow velocity of positive ions and
the electrostatic potential, respectively.
Suzuki \cite{M.S.1} showed the unique existence of solutions of 
the system \eqref{sp3} with the conditions
\begin{gather}
  \inf_{x\in\mathbb R_{+}}\rho(x)>0, \qu
  \phi(0)=\phi_b, \qu
  \lim_{x\rightarrow \infty}(\rho,u,\phi)(x)=(1,-u_{\infty},0), 
  \lb{sp4}
\end{gather}
\end{subequations}
where $u_{\infty}>0$ and $\phi_{b} \in \mathbb R$ are constants.
For more details, see Proposition \ref{EPsol} below.

It is of great interest to investigate the relation between 
the solutions of the Vlasov--Poisson system \eqref{sVP1} 
and the Euler--Poisson system \eqref{sp0}.
To clarify the relation,
we choose some approximate functions of the delta function $\delta(\xi-(-u_{\infty},0,0))$ for $f_{\infty}$,
and then show in Theorem \ref{DiffThm1} 
that the solution of \eqref{sVP1} converges to that of \eqref{sp0} 
by taking the limit of approximate functions to the delta function.
We call the limit {\it a delta mass limit}.

We review mathematical results on the Vlasov--Poisson system describing the motion of plasma.
For the Cauchy problem, early references \cite{BD,LP,K.P.,J.S.}
investigated the time-global solvability and dispersive analysis (see also textbooks \cite{glassey,G.R.}).  
Guo--Strauss \cite{GS1} established spatially periodic stationary solutions and also investigated its instability 
(see also \cite{GS2} studying the relativistic Vlasov--Maxwell system).
For the the initial--boundary value problem,
Guo \cite{Y.G.1} and Hwang--Vel\'azquez\cite{HV1} showed the time-global solvability 
adopting the specular reflection boundary condition.
Han-Kwan--Rousset \cite{HR} and Han-Kwan--Iacobelli \cite{HI1,HI2} 
analyzed the quasi-neutral limit of solutions satisfying a periodic boundary condition.
Furthermore, Skubachevskii \cite{A.S.1,A.S.2} and Skubachevskii--Tsuzuki \cite{ST1} 
focused on the analysis of toroidal magnetic plasma containment devices (tokamak), 
and established the existence of solutions of the initial--boundary value problem with an external magnetic field,
where supports $\supp f$ do not contact with boundaries.

The stationary problem has also been extensively studied.
In an infinite cylinder and a half-space,
Belyaeva \cite{Y.B.} and Skubachevskii \cite{A.S.3} constructed stationary solutions whose supports $\supp f$ do not contact with boundaries by applying an external magnetic field.
Similarly, Knopf \cite{P.K.1} established the stationary solution in the whole space,
where the support $\supp f$ is contained in an infinite cylinder.
Let us introduce results which treated more similar settings to this paper.
Greengard--Raviart \cite{GR1} and Rein \cite{G.R.2} constructed the stationary solutions 
in bounded domains adopting the Dirichlet and the specular refection boundary conditions, respectively. 
Recently, Esent\"urk--Hwang--Strauss \cite{EHS} showed the solvability of the stationary problem
with diffusive boundary conditions for various domains including a half-space.
For the half-space case, they considered a situation 
that there are only particles whose energies $|\xi|^{2}/2$ are bounded by some finite number,
and the boundary is not electrically charged.
This point is one of the differences between settings in \cite{EHS} and this paper.
We also emphasize that there is no research
which studies the delta mass limit mentioned above.

This paper is organized as follows.
Section \ref{S2} provides our main theorems on the solvability and delta mass limit.
In Section \ref{S3}, we treat the completely absorbing and attractive boundary, i.e. 
$f_{b}=\alpha=0$ and $\phi_{b}>0$.
Subsection \ref{S3.1} is devoted to discussion on a simpler derivation of the kinetic Bohm criterion \eqref{Bohm1}.
Subsection \ref{S3.2} establishes the solvability of the problem \eqref{sVP1}. 
We also justify the delta mass limit in subsection \ref{S3.3}.
In Section \ref{S4}, we study general boundaries, 
i.e. $(f_{b},\alpha)\neq (0,0)$ and $\phi_{b} \neq 0$.
 

Before closing this section, we give our notation used throughout this paper.

\medskip

\noindent
{\bf Notation.} 
For $ 1 \leq p \leq \infty$, $L^p(\Omega)$ is the Lebesgue space
equipped with the norm $\Vert\cdot\Vert_{L^{p}(\Omega)}$.
Let us denote by $(\cdot,\cdot)_{L^{2}(\Omega)}$ the inner product of $L^{2}(\Omega)$.
For $1<r<\infty$ and $-\infty \leq a<b \leq \infty$, the function spaces $L^{r}(a,b; L^{1}(\mathbb R^{2}))$  and $L^{r}_{loc}(\mathbb R; L^{1}(\mathbb R^{2}))$ are  defined by
\begin{align*}
L^{r}(a,b; L^{1}(\mathbb R^{2}))&:=\left\{ f \in L^{1}((a,b) \times \mathbb R^{2}) \left| \|f\|_{L^{r}(a,b;L^{1}(\mathbb R^{2}))}<\infty \right. \right\},
\\
L^{r}_{loc}(\mathbb R; L^{1}(\mathbb R^{2}))
&
:=\left\{ f \in L^{1}_{loc} (\mathbb R^{3}) \left|  \|f\|_{L^{r}(a,b; L^{1}(\mathbb R^{2}))}<\infty \ \text{for $-\infty < \forall a< \forall b < \infty$} \right.\right\},
\\
\|f\|_{L^{r}(a,b;L^{1}(\mathbb R^{2}))}&:= \left\{\int_{a}^{b} \left(\int_{\mathbb R^{2}} |f(\xi_{1},\xi')| d\xi' \right)^{r} d\xi_{1} \right\}^{1/r}.
\end{align*}
Furthermore, $\mathbb R_{+}:=\{x>0\}$ stands for a one-dimensional half space;
$\mathbb R_{+}^{3}:=\{\xi \in \mathbb R^{3} ; \xi_{1}>0 \}$ stands for a three-dimensional upper half space;
$\mathbb R_{-}^{3}:=\{\xi \in \mathbb R^{3} ; \xi_{1}<0 \}$ stands for a three-dimensional lower half space.
We also use the one-dimensional indicator function $\chi(s)$ of the set $\{s>0\}$.

\section{Main Results}\lb{S2}

We focus ourself on the analysis of solutions whose potential $\phi$ is monotone,
since the potential is observed as a monotone function when the plasma sheath is formed.
Let us give a definition of solutions of the boundary value problem \eqref{sVP1}.

\begin{defn}\label{DefS1}
We say that $(f,\phi)$ is a solution of 
the boundary value problem \eqref{sVP1} if it satisfies the following:
\begin{enumerate}[(i)]
\item $f \in L^{1}_{loc}(\overline{\mathbb R_{+}}\times \mathbb R^{3}) \cap C(\overline{\mathbb R_{+}};L^{1}(\mathbb R^{3}))$
and $\phi \in C^{1}(\overline{\mathbb R_{+}}) \cap C^{2}(\mathbb R_{+})$. 
\item $f(x)\geq 0$, and either $\D_{x}\phi(x)>0$ or $\D_{x}\phi(x)<0$.
\item $f$ solves 
\begin{subequations}\label{weak0}
\begin{gather}
({f},\xi_1\D_x\psi
+\D_x{\phi}\D_{\xi_1}\psi)_{L^2({\mathbb R}_+\times {\mathbb R}^3)}
+(f_b,\xi_1\psi(0,\cdot))_{L^2( {\mathbb R}_+^3)}
=0 \quad \hbox{for $\forall \psi \in {\cal X}$},  
\label{weak1}\\
\lim_{x\to \infty} \| f(x,\cdot)-f_{\infty} \|_{L^{1}(\mathbb R^{3})}=0,
\label{weak2}
\end{gather}
\end{subequations}
where 
${\cal X}:=\{f\in C_0^1(\overline{\mathbb{R}_+}\times \mathbb{R}^3)\;|\;
\hbox{$\alpha f(0,\xi_{1},\xi')=f(0,-\xi_1,\xi')$ for $(\xi_{1},\xi')\in {\mathbb R}_+^3$}\}$.
\item $\phi$ solves \eqref{seq2} with \eqref{sbc3} and \eqref{sbc4} in the classical sense.
\end{enumerate}
\end{defn}

The equation \eqref{weak1} is a standard weak form of the equation \eqref{seq1} and boundary condition \eqref{sbc1}.
We also remark that it is possible to replace {\it the classical sense} in the condition (iv) by {\it the weak sense}.
Indeed a weak solution $\phi$ of the problem of \eqref{seq2} with \eqref{sbc3} and \eqref{sbc4}
is a classical solution if $f \in C(\overline{\mathbb R_{+}};L^{1}(\mathbb R^{3}))$.

Next we discuss some necessary conditions for the solvability of the problem \eqref{sVP1}, 
which are used to state our main results.
To solve the Poisson equation \eqref{seq2} with \eqref{sbc4},
we must require the quasi-neutral condition
\begin{gather}\label{netrual1}
\int_{\mathbb R^{3}} f_{\infty}(\xi) d\xi = 1.
\end{gather}  
We remark that the boundary value problem \eqref{sVP1} is overdetermined. 
Let us explain briefly for the case $\phi_{b}=0$.
First $\phi$ must be zero,
and then the equation \eqref{seq1} implies that $f$ is independent of $x$.
This fact together with the boundary condition \eqref{sbc2} means that $f=f_{\infty}$.
On the other hand, due to the boundary condition \eqref{sbc1}, 
we have a necessary condition for $f_{\infty}$ and $f_{b}$:
\begin{gather*}
f_{\infty} (\xi) = f_{b}(\xi)+\alpha  f_{\infty}(-\xi_{1},\xi'), \quad \xi=(\xi_{1},\xi') \in {\mathbb R}^{3}_{+}.
\end{gather*}
Consequently, we cannot choose independently $f_{\infty}$ and $f_{b}$.
For the case $\phi_{b} \neq 0$, the \footnote{We may not find any physical meaning of \eqref{need3}--\eqref{need2},
but mathematically speaking there is no solution without them.}following are necessary conditions:
\begin{align}
&f_\infty(\xi)=f_b(\sqrt{\xi_1^2+2\phi_b},\xi')+\alpha f_\infty(-\xi_1,\xi'),  
\quad  \xi\in\R_+^3
& \text{if $\phi_{b}>0$},
\label{need3} \\
&f_\infty(\xi)=f_b(\sqrt{\xi_1^2+2\phi_b},\xi')+\alpha f_\infty(-\xi_1,\xi'), 
\quad  \xi\in(\sqrt{2|\phi_b|},\infty)\times\R^2,
\notag\\
&
f_\infty(\xi_1,\xi')=f_\infty(-\xi_1,\xi'), \quad 
\xi\in(-\sqrt{2|\phi_b|},\sqrt{2|\phi_b|})\times\R^2
& \text{if $\phi_{b}<0$}.
\label{need4}
\end{align}
In particular, for the completely absorbing and attractive boundary, i.e. $f_{b}=\alpha=0$ and $\phi_{b}>0$, it is written by
\begin{gather}
f_{\infty}(\xi)=0, \quad \xi_{1} > 0.
\label{need2} 
\end{gather}
We will show that \eqref{need3}, \eqref{need4}, and \eqref{need2} are necessary conditions in Lemmas \ref{lem41}, \ref{lem42}, and \ref{lem31}, respectively.
From the above observation, we also see that $(f,\phi)=(f_{\infty},0)$ 
is a unique solution of \eqref{sVP1} with $\phi_{b}=0$,
and hence suppose $\phi_{b}\neq 0$ hereafter.

We state our main results for the completely absorbing and attractive boundary, i.e. 
$f_{b}=\alpha=0$ and $\phi_{b}>0$ in subsection \ref{S2.1}. 
It is one of the most discussed situation in plasma physics.
Subsection \ref{S2.2} provides the main results for general boundaries, 
i.e. $(f_{b},\alpha)\neq (0,0)$ and $\phi_{b} \neq 0$.

\subsection{The Completely Absorbing and Attractive Boundary}\label{S2.1}

We first discuss the results of the solvability of the problem \eqref{sVP1} with 
the completely absorbing and attractive boundary.
After that we also study the delta mass limit of the solution.
It validates rigorously the relation of the kinetic and hydrodynamic Bohm criterions.

The solvability is summarized in Theorem \ref{existence1} below. 
Here the set $B$ is defined for $f_{\infty}$ as
\begin{align*}
B:=\{\varphi>0 \, ; \, V(\phi)>0 \ \text{for} \ \phi \in (0,\varphi] \},
\end{align*}
where
\begin{align}\label{V}
V(\phi):=\int_{0}^{\phi} \rho_{i}(\varphi) - n_{e}(\varphi)  \, d\vphi, \quad
\rho_{i}(\phi):=\int_{\mathbb R^{3}} f_{\infty}(\xi)\frac{-\xi_{1}}{\sqrt{\xi_{1}^{2}+2\phi}} d\xi.
\end{align}
The function $\rho_{i}$ is well-defined for $f_{\infty} \in L^{1}(\mathbb R)$. Indeed,
\begin{gather}\label{rho0}
|\rho_{i}(\phi)| \leq \|f_{\infty}\|_{ L^{1}(\mathbb R^{3})} \quad
\text{for $\phi\geq 0$.}
\end{gather}

\begin{thm}\label{existence1}
Let $f_{b}=\alpha=0$ and $\phi_{b}>0$.
Suppose that $f_{\infty} \in L^{1}(\mathbb R^{3})$ satisfies $f_{\infty} \geq 0$ and
the necessary conditions \eqref{netrual1} and \eqref{need2}.
\begin{enumerate}[(i)]
\item Assume that 
\begin{gather}
\int_{\mathbb R^{3}}\xi_{1}^{-2}f_{\infty}(\xi) d\xi <1.
\label{Bohm2}
\end{gather}
Then the set $B$ is not empty. 
Furthermore, if and only if $\phi_{b}<\sup B$ holds, the problem \eqref{sVP1} has a unique solution $(f,\phi)$.
There also hold that 
\begin{gather}
f(x,\xi)=f_{\infty}(-\sqrt{\xi_{1}^{2}-2\phi(x)},\xi')\chi(\xi_{1}^{2}-2\phi(x))\chi(-\xi_{1}),
\label{fform} \\
|\D_{x}^{l} \phi(x)| \leq C e^{- c x} 
\quad \text{for $l=0,1,2$},
\label{decay1}
\end{gather}
where $\chi(s)$ is the one-dimensional indicator function of the set $\{s>0\}$, and 
$c$ and $C$ are positive constants independent of $x$. 
In addtion, if $f_{\infty} \in C^{2}(\mathbb R^{3})$, 
then the solution $(f,\phi) \in C^{1}(\overline{\mathbb R_{+}}\times \mathbb R^{3}) \times C^{2}(\overline{\mathbb R_{+}})$ is a classical solution.

\item Assume that 
\begin{gather*}
\int_{\mathbb R^{3}}\xi_{1}^{-2}f_{\infty}(\xi) d\xi =1, \quad B\neq \emptyset.
\end{gather*}
If and only if $\phi_{b}<\sup B$ holds, the problem \eqref{sVP1} has a unique solution $(f,\phi)$.
Further, \eqref{fform} holds. If $f_{\infty} \in C^{2}(\mathbb R^{3})$, 
then the solution $(f,\phi) \in C^{1}(\overline{\mathbb R_{+}}\times \mathbb R^{3}) \times C^{2}(\overline{\mathbb R_{+}})$ is a classical solution.

\item Assume that 
\begin{gather*}
\int_{\mathbb R^{3}}\xi_{1}^{-2}f_{\infty}(\xi) d\xi =1, \quad B= \emptyset.
\end{gather*}
Then the problem \eqref{sVP1} admits no solution.

\item Assume \footnote{This includes the case $\int_{\mathbb R^{3}}\xi_{1}^{-2}f_{\infty}(\xi) d\xi=\infty$.}that 
\begin{gather}
\int_{\mathbb R^{3}}\xi_{1}^{-2}f_{\infty}(\xi) d\xi >1.
\label{notBohm0}
\end{gather}
Then the problem \eqref{sVP1} admits no solution.

\end{enumerate}
\end{thm}

This theorem covers all possible cases of $\phi_{b}>0$ and $f_{\infty} \in L^{1}(\mathbb R^{3})$ with $f_{\infty} \geq 0$, 
and clarifies completely when there is a solution or not.
Assertion (iv) means that the kinetic Bohm criterion \eqref{Bohm1} is a necessary condition for 
the solvability of the problem \eqref{sVP1}, since \eqref{notBohm0} is the negation of \eqref{Bohm1}.
Assertion (iv) will be shown in subsection \ref{S3.1}.
The proof gives simultaneously a simpler derivation of the kinetic Bohm criterion than that of \cite{K.R.1}.
Subsection \ref{S3.1} also provides the proof of Assertion (iii).
In subsection \ref{S3.2}, we will prove Assertions (i) and (ii).

Next we discuss the relation between 
the solutions of the Vlasov--Poisson system \eqref{sVP1} 
and the Euler--Poisson system \eqref{sp0}.
As mentioned in Section \ref{S1}, the hydrodynamic Bohm criterion \eqref{Bohm4} 
can be derived from the kinetic Bohm criterion \eqref{Bohm1} 
by plugging a delta function $\delta(\xi-(-u_{\infty},0,0))$ into $f_{\infty}$.
To investigate the relation of solutions, 
we use 
\begin{equation}\label{fve}
f_{\infty,\ve}(\xi):=\frac{1}{\ve^{3}}\vphi \left(\frac{\xi-(-u_{\infty},0,0)}\ve\right),
\end{equation}
where $u_{\infty}$ is a positive constant and 
\begin{gather}\label{vphi0}
0<\ve<1, \quad \vphi \in C_{0}^{\infty}(\mathbb R^{3}), \quad \vphi(\xi) \geq 0, \quad 
{\rm supp} \vphi \subset \{|\xi| < 1 \}, \quad \int_{{\mathbb R}^{3}} \vphi(\xi) d\xi=1. 
\end{gather}
Note that 
\begin{gather*}
\lim_{\ve\to 0}\int_{{\mathbb R}^{3}}  f_{\infty,\ve}(\xi) \psi(\xi) d\xi = \psi(-u_{\infty},0,0), \quad
\forall \psi \in C(\mathbb R^{3}).
\end{gather*}
This means that $f_{\infty,\ve}(\xi)$ converges to the delta function $\delta(\xi-(-u_{\infty},0,0))$
as $\ve \to 0$.
Let us denote by $(f_{\ve},\phi_{\ve})$ 
the solution of the problem \eqref{sVP1} with $f_{\infty}=f_{\infty,\ve}$,
where $f_{\infty,\ve} \in C_{0}^{\infty}(\mathbb R^{3})$ satisfies
the conditions \eqref{netrual1}, \eqref{need2}, and \eqref{Bohm2} 
being in Theorem \ref{existence1} for $u_{\infty}>1$ and $\ve \ll 1$.
We also introduce the moments 
\begin{gather*}
\rho_{\ve}(x):=\int_{{\mathbb R}^{3}} f_{\ve}(x,\xi) d\xi , \quad u_{\ve}(x):= \frac{1}{\rho_{\ve}(x)}\int_{{\mathbb R}^{3}} \xi_{1} f_{\ve} (x,\xi) d\xi.
\end{gather*}
It is expected that the moments converge to $(\rho,u)$ of the Euler--Poisson system \eqref{sp0}
by taking the delta mass limit as $\ve \to 0$.
The solvability of the problem \eqref{sp0} is summarized in the following proposition.

\begin{pro}[\!\!\cite{M.S.1}]\label{EPsol}
Let \eqref{Bohm4} and $\phi_{b} \in (-u_{\infty}^{2}/2,\infty)$ hold. 
Then the problem \eqref{sp0} has a unique monotone solution 
$(\rho,u,\phi) \in C^{\infty}(\overline{{\mathbb R}_{+}})$.
Furthermore,  there hold that
\begin{gather}
\rho=\frac{u_{\infty}}{\sqrt{u_{\infty}^{2}+2\phi}}, \quad 
u=-\sqrt{u_{\infty}^{2}+2\phi}, \quad 
\D_{x}\phi(x) \neq 0,
\label{EPform1} \\
|\D_{x}^{l}(\rho-1,u+u_{\infty},\phi)(x)| \leq Ce^{-cx}, \quad l=0,1,2,
\notag 
\end{gather}
where $c$ and $C$ are positive constants independent of $x$. 
\end{pro}

We are now in a position to state our results on the delta mass limit. 
\begin{thm}\label{DiffThm1}
Let $f_{b}=\alpha=0$, $n_{e}(\phi)=e^{-\phi}$, and $u_{\infty}>1$.
There exist positive constants $\delta_{0}$ and $C_{0}$ such that if $0<\phi_{b}<\delta_{0}$, then the following holds:
\begin{gather}\label{conver0}
\sup_{x \in {\mathbb R}_{+}} |\rho_{\ve}(x) - \rho(x)| 
+\sup_{x \in {\mathbb R}_{+}} |\rho_{\ve}u_{\ve}(x) - \rho u(x)|
+\sup_{x \in {\mathbb R}_{+}} |\phi_{\ve}(x) - \phi(x)| \leq C_{0}\ve
\end{gather}
for any $\ve \in (0,\ve_{0})$, where $\ve_{0} := (u_{\infty}-1)/2$.
\end{thm}

Now we mention some remarks.
We cannot choose directly a Maxwellian for $f_{\infty}$ due to 
the necessary condition \eqref{need2} for the solvability of the problem \eqref{sVP1}, 
but it allows products of the Maxwellian and cut-off functions.
For such products, the delta mass limit as $\ve\to0$ corresponds to that 
the temperature of the Maxwellian tends to zero.
In this sense, it is reasonable that the limit is a solution of 
the Euler--Poisson system of cold plasma.

\subsection{General Boundaries}\lb{S2.2}
This section provides the main results for general boundary conditions, i.e.
$(f_{b},\alpha)\neq (0,0)$ and $\phi_{b} \neq 0$.
Similarly as Theorem \ref{existence1}, we use notation
\begin{align}
&B^{+}:=\{\varphi>0 \, ; \, V^{+}(\phi)>0 \ \text{for} \ \phi \in (0,\varphi] \},
&V^{+}(\phi):=\int_{0}^{\phi} \left(\rho_{i}^{+}(\varphi) - n_{e}(\varphi) \right)  d\vphi,
\label{V+}\\
&B^{-}:=\{\varphi<0 \, ; \, V^{-}(\phi)>0 \ \text{for} \ \phi \in [\varphi,0) \},
&V^{-}(\phi):=\int_{0}^{\phi} \left(\rho_{i}^{-}(\varphi) - n_{e}(\varphi) \right)  d\vphi,
\label{V-}
\end{align}
where
\begin{align}
\rho_{i}^{+}(\phi)
&:=\int_{{\mathbb R}^3}f_\infty(\xi)\frac{|\xi_1|}{\sqrt{\xi_1^2+2\phi}}\,d\xi
\notag \\
&\quad +\frac{2 }{1-\alpha}\int_{\sqrt{2(\phi_b-\phi)\chi(\phi_b-\phi)}}^{\sqrt{2\phi_b}}\int_{\mathbb R^2}
f_b(\xi)\frac{\xi_1}{\sqrt{\xi_1^2+2(\phi-\phi_b)}}\,d\xi_1\,d\xi' &
\text{for $\phi \geq 0$},
\label{rho++}\\
\rho_{i}^{-}(\phi)
&:=\int_{{\mathbb R}^3}f_\infty(\xi)\frac{|\xi_1|}{\sqrt{\xi_1^2+2\phi}}
\chi(\xi_1^2+2\phi)\,d\xi &
\text{for $\phi \leq 0$}.
\label{rho--}
\end{align}

The functions $\rho_{i}^{\pm}$ are well-defined for $f_{b} \in L^{1}(\mathbb R^{3}_{+})$ and
$f_{\infty} \in L^{1}(\mathbb R^{3})$ with $f_{b},f_{\infty} \geq 0$, since all the integrants are nonnegative. 
Furthermore, for $r>2$ and $M>0$, it is seen \footnote{If the problem \eqref{sVP1} has a solution $(f,\phi)$, then $\rho_{i}^{+}\in C([0,\phi_{b}])$ and $\rho_{i}^{-}\in C([\phi_{b},0])$ must hold, whether or not $f_b\in L^{r}(0,\sqrt{2\phi_{b}};L^{1}(\mathbb R^{2}))$ and $f_{\infty} \in L^{r}_{loc}(\mathbb R;L^{1}(\mathbb R^{2}))$ hold. For more details, see Lemmas \ref{lem41} and \ref{lem42}.}that 
\begin{align}
|\rho_{i}^{+}(\phi)| & \leq \|f_{\infty}\|_{L^{1}(\mathbb R^{3})} + C\|f_{b}\|_{L^{r}(0,\sqrt{2\phi_{b}};L^{1}(\mathbb R^{2}))} & \text{for }  & \phi \geq 0,
\label{rho+}\\
|\rho_{i}^{-}(\phi)|  &\leq  \sqrt{2} \|f_{\infty}\|_{L^{1}(\mathbb R^{3})} +C_{M}\|f_{\infty}\|_{L^{r}(-2\sqrt{M},2\sqrt{M};L^{1}(\mathbb R^{2}))} & \text{for } & \phi \in [-M,0],
\label{rho-}
\end{align}
where $C$ is a positive constant, and $C_{M}$ is a positive constant depending on $M$. 
The proofs of the estimates are postponed until  Appendix \ref{C}. 
In the definition \eqref{rho++}, the indicator function $\chi(\phi_b-\phi)$ is used 
to extend continuously $\rho_{i}^{+}$ beyond $\phi=\phi_{b}$. 

Now we state main theorems on the solvability by employing the function spaces $ L^{r}(a,b;L^{1}(\mathbb R^{2}))$ and $L^{r}_{loc}(\mathbb R;L^{1}(\mathbb R^{2}))$. Here the set $B^{+}$ depends on $\phi_{b}$ whereas the set $B^{-}$ is independent of $\phi_{b}$.

\begin{thm} \label{existence2}
{\rm (Attractive boundary)} \ Let $\phi_b>0$, $f_b\in L^1(\R_+^3) \cap L^{r}(0,\sqrt{2\phi_{b}};L^{1}(\mathbb R^{2}))$,  \footnote{We can construct multiple solutions for the case $\alpha = 1$, and therefore we exclude it.}$\alpha \neq 1$, $f_{\infty} \in L^{1}(\mathbb R^{3})$, and $f_{b},f_{\infty} \geq 0$ for some $r>2$.
Suppose \footnote{If a solution exists, then \eqref{netrual1}, \eqref{need3}, and $V^{+} \in C^{1}([0,\phi_{b}])$ must hold. For more details, see Lemma \ref{lem41}.}that \eqref{netrual1}, \eqref{need3}, and $V^{+} \in C^{2}([0,\phi_{b}])$ hold. 

\begin{enumerate}[(i)]
\item Assume that $d^{2} V^{+}/d\phi^{2}(0)>0$.
Then the set $B^{+}$ is not empty. 
Furthermore, \footnote{If $\supp f_{b} \subset (c,\infty)\times \mathbb R^{2}$ for some $c>0$, the second term in the definition of $\rho_{i}^{+}$ vanishes for $\phi_{b}<c^{2}/2$. This means that $B^{+}$ is independent of $\phi_{b}$. In this case, $\phi_{b}<\sup B^{+}$ holds for $\phi_{b} \ll 1$.}if and only if $\phi_{b}<\sup B^{+}$ holds, the problem \eqref{sVP1} has a unique solution $(f,\phi)$.
There also hold that 
\begin{gather}
\begin{aligned}
{f}(x,\xi)
& =f_\infty(-\sqrt{\xi_1^2-2{\phi}(x)},\xi')
\chi(\xi_{1}^{2}-2\phi(x))\chi(-\xi_{1}) \\
& \qquad
+\frac{1}{1-\alpha}\,f_b(\sqrt{\xi_1^2-2{\phi}(x)+2\phi_b},\xi')
\chi(-\xi_1^2+2{\phi}(x)) \\
& \qquad
+f_\infty(\sqrt{\xi_1^2-2{\phi}(x)},\xi')
\chi(\xi_{1}^{2}-2\phi(x))\chi(\xi_{1}),
\end{aligned}
\label{fform2} \\
|\D_{x}^{l} \phi(x)| \leq C e^{- c x}, \quad l=0,1,2, 
\label{decay2}
\end{gather}
where $c$ and $C$ are positive constants independent of $x$. 
\item Assume that $d^{2} V^{+}/d\phi^{2}(0)=0$ and $B^{+}\neq \emptyset$.
If and only if $\phi_{b}<\sup B^{+}$ holds, the problem \eqref{sVP1} has a unique solution $(f,\phi)$.
Furthermore, \eqref{fform2} holds. 
\item Assume that $d^{2} V^{+}/d\phi^{2}(0)=0$ and $B^{+}= \emptyset$.
Then the problem \eqref{sVP1} admits no solution.
\item Assume that $d^{2} V^{+}/d\phi^{2}(0)<0$. 
Then the problem \eqref{sVP1} admits no solution.
\end{enumerate}

\end{thm}

\begin{thm} \label{existence3}
\noindent
{\rm (Repulsive boundary)} \ Let $\phi_b<0$, $f_b\in L^1(\R_+^3)$, $f_{\infty} \in L^{1}(\mathbb R^{3}) \cap L^{r}_{loc}(\mathbb R;L^{1}(\mathbb R^{2}))$, and $f_{b},f_{\infty} \geq 0$ for some $r>2$.
Suppose \footnote{If a solution exists, then \eqref{netrual1}, \eqref{need4}, and $V^{-} \in C^{1}([\phi_{b},0])$ must hold.
For more details, see Lemma \ref{lem42}.}that \eqref{netrual1}, \eqref{need4}, and $V^{-} \in C^{2}([\phi_{b},0])$ hold.
\begin{enumerate}[(i)]
\item Assume that $d^{2} V^{-}/d\phi^{2}(0)>0$.
Then the set $B^{-}$ is not empty. 
Furthermore, 
if and only if $\phi_{b}>\inf B^{-}$ holds, the problem \eqref{sVP1} has a unique solution $(f,\phi)$.
There also hold that \eqref{decay2} and
\begin{gather}
\begin{aligned}
f(x,\xi)=f_{\infty}(-\sqrt{\xi_{1}^{2}-2\phi(x)},\xi')\chi(-\xi_{1})+f_{\infty}(\sqrt{\xi_{1}^{2}-2\phi(x)},\xi')\chi(\xi_{1}).
\end{aligned}
\label{fform5}
\end{gather}
\item Assume that $d^{2} V^{-}/d\phi^{2}(0)=0$ and $B^{-}\neq \emptyset$.
If and only if $\phi_{b}>\inf B^{-}$ holds, the problem \eqref{sVP1} has a unique solution $(f,\phi)$.
Furthermore, \eqref{fform5} holds. 
\item Assume that $d^{2} V^{-}/d\phi^{2}(0)=0$ and $B^{-}= \emptyset$.
Then the problem \eqref{sVP1} admits no solution.
\item Assume that $d^{2} V^{-}/d\phi^{2}(0)<0$. 
Then the problem \eqref{sVP1} admits no solution.
\end{enumerate}

\end{thm}

From these theorems, we conclude that
the following conditions are general representations of the kinetic Bohm criterion:
\begin{gather}\label{gBohm0}
\frac{d^{2} V^{\pm}}{d\phi^{2}}(0)\geq 0.
\end{gather}
Indeed, it can be rewritten by the kinetic Bohm criterion \eqref{Bohm1}
for the case $f_{b}=0$, $\alpha=0$, and $\phi_{b}>0$ (for details, see \eqref{V2}).
We can find some functions $(f_{b},f_{\infty})$ so that $V^{\pm} \in C^{2}$ and \eqref{gBohm0} hold, 
for example, the functions $(g_{b,\ve},g_{\infty,\ve})$ defined in \eqref{fb2} and \eqref{fve2} below.
We introduce some more general functions in Appendix \ref{B0}.
Theorems \ref{existence2} and \ref{existence3} will be shown in subsections \ref{S4.1} and \ref{S4.2}, respectively.

We also state the results on the delta mass limit.
Let us define the functions \footnote{
We first determine $g_{\infty,\ve}$ as \eqref{fve2},
and then find a suitable $g_{b,\ve}$ as \eqref{fb2}
so that the necessary conditions \eqref{need3} and \eqref{need4} hold.
This choice of $g_{\infty,\ve}$ is 
one of the simplest extensions of $f_{\infty,\ve}$ in \eqref{fve}.}$g_{b,\ve}$ and $g_{\infty,\ve}$ as
\begin{align}
g_{b,\ve}(\xi)&:=\frac{m_{b}}{\ve^{3}}\vphi\left(\frac{(\sqrt{\xi_{1}^{2}-2\phi_{b}},\xi')-(v_{b},0,0)}{\ve}\right) \chi(\xi_{1}^{2}-2\phi_{b}), \quad \xi \in {\mathbb R}^{3}_{+},
\label{fb2} \\
g_{\infty,\ve}(\xi)&:=\frac{m_{\infty}}{\ve^{3}}\vphi \left(\frac{\xi-(-v_{\infty},0,0)}\ve\right)
+\frac{m_{b}}{\ve^{3}}\vphi\left(\frac{\xi-(v_{b},0,0)}{\ve}\right)
\notag \\
&\qquad +\alpha\frac{m_{\infty}}{\ve^{3}}\vphi \left(\frac{(-\xi_{1},\xi')+(v_{\infty},0,0)}\ve\right), \quad
\xi \in {\mathbb R}^{3},
\label{fve2}
\end{align}
where $\ve$ and $\varphi$ are defined in \eqref{vphi0}, 
and $m_{b}\geq 0$, $m_{\infty}>0$, $v_{b}> 0$, and $v_{\infty}>0$ are constants. 
It is supposed that 
\begin{gather}
m_{b}+(1+\alpha)m_{\infty}=1, \quad
m_{b}v_{b}^{-2}+(1+\alpha)m_{\infty}v_{\infty}^{-2}<1,
\label{velocity1}
\end{gather}
which ensures \eqref{netrual1} and $\int_{\mathbb R^{3}}\xi_{1}^{-2}g_{\infty,\ve}(\xi) d\xi < 1$ for $\ve \ll 1$.
Let us denote by $(f_{\ve},\phi_{\ve})$ 
the solution of the problem \eqref{sVP1} with $(f_{b},f_{\infty})=(g_{b,\ve},g_{\infty,\ve})$.
The important thing to note here is that $g_{b,\ve}$ and $g_{\infty,\ve}$ satisfy
the conditions \eqref{netrual1}--\eqref{need4} and $d^{2} V^{\pm}/d\phi^{2}(0)=-\int_{\mathbb R^{3}}\xi_{1}^{-2}g_{\infty,\ve}(\xi) d\xi+1>0$
being in Theorems \ref{existence2} and \ref{existence3} for the case $|\phi_{b}| \ll 1$ and $\ve\ll 1$.
In this case, $\sup B^{+}$ is independent of $\phi_{b}$, and also $\phi_{b}<\sup B^{+}$ holds, 
since the second term in the definition \eqref{rho++} of $\rho_{i}^{+}$ vanishes.
Note that $\inf B^{-}$ is always independent of $\phi_{b}$.
We also introduce the moments 
\begin{gather*}
\rho_{\ve}(x):=\int_{{\mathbb R}^{3}} f_{\ve}(x,\xi) d\xi , \quad u_{\ve}(x):= \frac{1}{\rho_{\ve}(x)}\int_{{\mathbb R}^{3}} \xi_{1} f_{\ve} (x,\xi) d\xi.
\end{gather*}

\begin{thm}\label{DiffThm2}
Let $\phi_{b}\neq 0$, $\alpha \neq 1$, and \eqref{velocity1} hold.
There exist positive constants $\ve_{0}$, $\delta_{0}$, and $C_{0}$ such that if 
$|\phi_{b}|<\delta_{0}$, then the following holds:
\begin{gather}\label{conver5}
\sup_{x \in {\mathbb R}_{+}} |\rho_{\ve}(x) - \rho(x)| 
+\sup_{x \in {\mathbb R}_{+}} |\rho_{\ve}u_{\ve}(x) - \rho u(x)|
+\sup_{x \in {\mathbb R}_{+}} |\phi_{\ve}(x) - \phi(x)| \leq C_{0}\ve
\end{gather}
for any $\ve \in (0,\ve_{0})$, where $\phi=\phi(x)$ solves a boundary value problem
\begin{subequations}\label{Hphi0}
\begin{gather}
\D_{xx}\phi=
\frac{m_{b}v_{b}}{\sqrt{v_{b}^{2}+2\phi}}
+\frac{(1+\alpha)m_{\infty}v_{\infty}}{\sqrt{v_{\infty}^{2}+2\phi}}
-n_{e}(\phi), \quad x>0,
\label{Hphi1}\\
\phi(0)=\phi_{b}, \quad \lim_{x\to 0}\phi(x)=0,
\label{Hphi2}
\end{gather}
\end{subequations}
and $\rho=\rho(x)$ and $u=u(x)$ are defined as 
\begin{gather}\label{Hrho}
\rho(x):=\frac{m_{b}v_{b}}{\sqrt{v_{b}^{2}+2\phi(x)}}+
\frac{(1+\alpha)m_{\infty}v_{\infty}}{\sqrt{v_{\infty}^{2}+2\phi(x)}},
 \quad
u(x):=\frac{m_{b}v_{b}+(\alpha-1)m_{\infty}v_{\infty}}{\rho(x)}.
\end{gather}
\end{thm}

For the case $m_{b}=0$, i.e. $f_{b}=g_{b,\ve}=0$, the functions $(\rho,u,\phi)$ 
in Theorem \ref{DiffThm2} solve a problem
\begin{subequations}\label{Hmodel}
\begin{gather}
(\rho u)'=0, \quad 
\frac{(\alpha+1)^{2}}{(\alpha-1)^{2}}uu'=\phi', \quad 
\phi''=\rho-n_{e}(\phi), \quad 
x>0,
\\
  \inf_{x\in\mathbb R_{+}}\rho(x)>0, \qu
  \lim_{x\rightarrow \infty}(\rho,u,\phi)(x)=\left(1,\frac{\alpha-1}{\alpha+1}v_{\infty},0\right), \qu
  \phi(0)=\phi_b,
\end{gather}
\end{subequations}
which is similar to \eqref{sp0}. 
We conclude that \eqref{Hmodel} is a hydrodynamic model taking the reflection of positive ions
on the boundary into account.

\section{The Completely Absorbing and Attractive Boundary}\lb{S3}
In this section, we study the solvability and delta mass limit in Theorems \ref{existence1} and \ref{DiffThm1} for completely absorbing and attractive boundary, i.e. $f_{b}=\alpha=0$ and $\phi_{b}>0$.
We first show Assertion (iv) in Theorem \ref{existence1} in subsection \ref{S3.1}.
The proof gives simultaneously a simpler derivation of the kinetic Bohm criterion \eqref{Bohm1} than that of \cite{K.R.1},
and also shows that \eqref{need2} is a necessary condition.
In subsection \ref{S3.1}, we also prove Assertion (iii).
Subsections \ref{S3.2} deals with the proofs of  Assertions (i) and (ii) of Theorem \ref{existence1}.
We show Theorem \ref{DiffThm1} in subsection \ref{S3.3}.

\subsection{A Derivation of the Kinetic Bohm Criterion}\label{S3.1}
We discuss necessary conditions for the solvability 
by assuming that the solution $(f,\phi)$ of the problem \eqref{sVP1} with $f_{b}=\alpha=0$ and $\phi_{b}>0$ exists.
First we show the following lemma 
which ensures immediately Assertion (iv) in Theorem \ref{existence1}.
Note that \eqref{Bohm1} is the negation of \eqref{notBohm0}.

\begin{lem}\label{lem31}
Let $f_{b}=\alpha=0$, $\phi_{b}>0$, $f_{\infty} \in L^{1}(\mathbb R^{3})$, $f_{\infty} \geq 0$, and \eqref{netrual1} hold.
Suppose that  the problem \eqref{sVP1} has a solution $(f,\phi)$.
Then the function $f_{\infty}$ satisfies the condition \eqref{need2}
and the kinetic Bohm criterion \eqref{Bohm1}.
Furthermore, $\phi$ solves \eqref{phieq2} below with \eqref{sbc3} and \eqref{sbc4},
and $f$ is written as \eqref{fform} by $\phi$.
\end{lem} 

\begin{proof}
First $\D_{x} \phi(x) <0$ holds thanks to $\phi_{b}>0$ and the condition (ii) in Definition \ref{DefS1}. 
Owing to \eqref{sbc4}, \eqref{weak2}, 
$f \in C(\overline{\mathbb R_{+}}; L^{1}(\mathbb R^{3}))$, and $\phi \in C(\overline{\mathbb R_{+}})$, 
it follows from \eqref{seq2} that $\D_{xx} \phi$ is bounded and therefore $\D_{x} \phi$ is uniformly continuous on $\overline{\mathbb R_{+}}$.
Then the fact together with \eqref{sbc4} and $\D_{x} \phi(x)<0$ leads to a necessary condition
\begin{gather}\label{lim1}
\lim_{x \to \infty}\D_{x}\phi(x)=0.
\end{gather}

Let us first show that $f_{\infty}$ satisfies \eqref{need2} and $f$ is written as \eqref{fform}.
Regarding $\phi$ as a given function and then applying the characteristics method to \eqref{weak1}, 
we see that the value of $f$ must remain a constant along the following characteristic curves:
\begin{gather*}
\frac{1}{2}\xi_{1}^{2}-\phi(x)=c,
\end{gather*}
where $c$ is some constant. We draw the illustration of characteristics in Figure \ref{fig1} below.
\begin{figure}[htbp]
\begin{center}
   \includegraphics[width=0.5\textwidth]{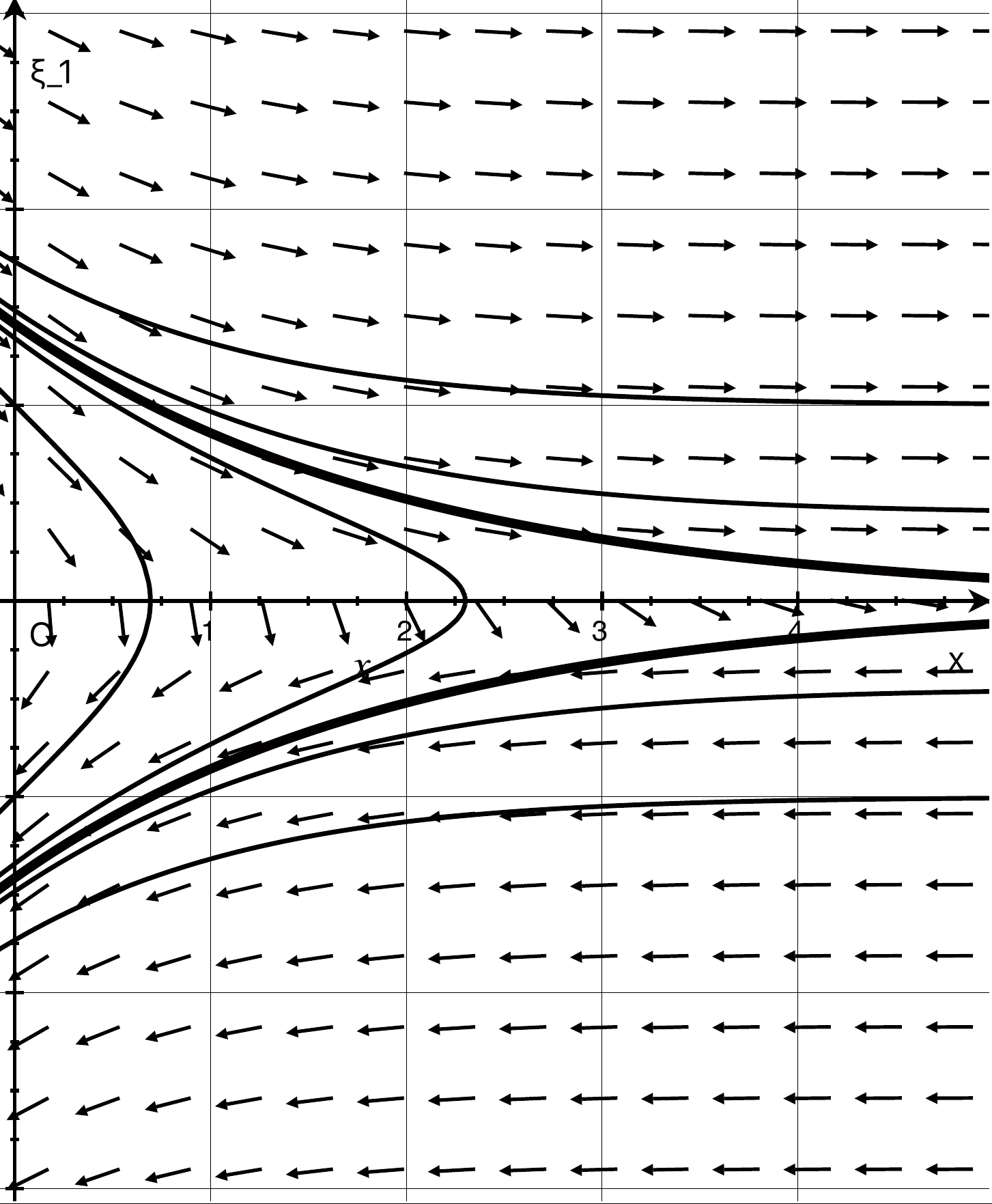}
\end{center}
  \caption{characteristic curves for the case $\partial_{x} \phi <0$}  
  \label{fig1}
\end{figure}
It tells us that
\begin{subequations}\label{chara0}
\begin{align}
f(y,-\sqrt{\eta_{1}^{2}+2\phi(y)},\eta')=f_{\infty}(\eta), &\quad (y,\eta_{1},\eta') \in X_{1},
\label{chara1}
 \\
f(y,\pm\sqrt{2\phi(y)-\eta_{1}^{2}},\eta')=0, &\quad (y,\eta_{1},\eta') \in X_{2},
\label{chara2}
 \\
f(y,\sqrt{\eta_{1}^{2}+2\phi(y)},\eta')=f_{\infty}(\eta)=0, &\quad  (y,\eta_{1},\eta') \in X_{3},
\label{chara3}
\end{align}
\end{subequations}
where $\eta=(\eta_{1},\eta_{2},\eta_{3})=(\eta_{1},\eta')$ and
\[
X_{1}:=\overline{\R_+}\times\R_-^3, \quad
X_{2}:=\{(y,\eta_1,\eta')\in\overline{\R_+}\times\R^3\;|\;\eta_1^2<2{\phi}(y)\leq 2\phi_b\}, \quad
X_{3}:=\overline{\R_+}\times\R_+^3.
\]
The last inequality \eqref{chara3} means that \eqref{need2} must hold. 
Furthermore, we conclude from \eqref{chara0} that $f$ must be written by \eqref{fform}, i.e.
\begin{gather*}
f(x,\xi)=f_{\infty}(-\sqrt{\xi_{1}^{2}-2\phi(x)},\xi')\chi(\xi_{1}^{2}-2\phi(x))\chi(-\xi_{1}),
\end{gather*}
where $\chi(s)$ is the one-dimensional indicator function of the set $\{s>0\}$.

Next we show that $\phi$ satisfies \eqref{phieq2} below.
Integrating $\eqref{fform}$ over $\mathbb R^{3}$, 
and using \eqref{need2} and the change of variable $\sqrt{\xi_{1}^{2}-2\phi}=-\zeta_{1}$, 
we see that
\begin{gather}
\int_{\mathbb R^{3}} f(x,\xi) \,d\xi = \int_{\mathbb R^{3}} f_{\infty}(\xi)\frac{-\xi_{1}}{\sqrt{\xi_{1}^{2}+2\phi(x)}} d\xi = \rho_{i}(\phi(x)),
\label{rho0'}
\end{gather}
where $\rho_{i}$ is the same function defined in \eqref{V}.
Substituting \eqref{rho0'} into \er{seq2} yields an ordinary differential equation for $\phi$:
\begin{equation}\lb{phieq1} 
\partial_{xx} \phi 
= \rho_{i}(\phi) 
- n_{e}(\phi), \quad x>0.
\end{equation}
Multiply \er{phieq1} by $\D_{x}\phi$, integrate it over $(x,\infty)$, and use \eqref{sbc4} and \eqref{lim1} to obtain
\begin{equation}\lb{phieq2}
(\D_{x} \phi)^2=2V(\phi), \quad V(\phi)=\int_{0}^{\phi} \left(\rho_{i}(\varphi) - n_{e}(\varphi) \right)  d\vphi,
\end{equation}
where $V$ is the same function defined in \eqref{V}.
Thus $\phi$ must satisfy \eqref{phieq2}.

To obtain the kinetic Bohm criterion \eqref{Bohm1},
we divide the proof into two cases $\int_{\mathbb R^{3}}\xi_{1}^{-2}f_{\infty}(\xi) d\xi$ $<\infty$ and
$\int_{\mathbb R^{3}}\xi_{1}^{-2}f_{\infty}(\xi) d\xi=\infty$.
For the former case, $V \in C^{2}([0,\phi_{b}])$ follows from the monotone convergence theorem
and the fact $f_{\infty},\xi_{1}^{-2}f_{\infty} \in L^{1}(\mathbb R^{3})$.
We also see from \eqref{netrual1} and \eqref{need2} that
\begin{equation}\lb{V2}
V(0)=\frac{d V}{d \phi}(0)=0, \quad
\frac{d^{2} V}{d \phi^{2}}(0) 
=-\int_{\mathbb R^{3}}\xi_{1}^{-2}f_{\infty}(\xi) d\xi+1.
\end{equation}
Hence, we arrive at the criterion \eqref{Bohm1} from  
\eqref{phieq2} and \eqref{V2} with the aid of the Taylor theorem.
Furthermore, we show that the other case does not occur.
Suppose that $\int_{\mathbb R^{3}}\xi_{1}^{-2}f_{\infty}(\xi) d\xi=\infty$ holds.
Then there hold that $V \in C^{2}((0,\phi_{b}])$, $V(0)=\frac{d V}{d \phi}(0)=0$, and
\begin{gather*}
\frac{d^{2} V}{d \phi^{2}}(\phi) <-C_{0} \quad \text{for $\phi \in (0,c_{0}]$},
\end{gather*}
where $c_{0}$ and $C_{0}$ are some positive constants.
These facts imply that $V(\phi)<0$ holds for $\phi \in (0,c_{0}]$. 
It contradicts to \eqref{phieq2}, and thus $\int_{\mathbb R^{3}}\xi_{1}^{-2}f_{\infty}(\xi) d\xi<\infty$ holds.
Consequently, \eqref{Bohm1} must hold.
\end{proof}

This proof provides a simpler derivation of the kinetic Bohm criterion \eqref{Bohm1}.

Next we show Assertion (iii).

\begin{proof}[Proof of Assertion (iii) in Theorem \ref{existence1}]
In the case $B\neq \emptyset$, either one of the following holds:
\begin{enumerate}[(a)]
\item $V(\phi)<0$ around $\phi=0$.  
\item There exists a sequence $\{\phi_{n}\}_{n=1}^{\infty}$ such that $V(\phi_{n})=0$ and $\lim_{n\to \infty} \phi_{n}=0$.
\end{enumerate}
On the other hand, it is clear from \eqref{phieq2} that the case (a) does not occur.  
It suffices to consider the case (b). 
Suppose that a solution $(f,\phi)$ of the problem \eqref{sVP1} exists. 
For sufficiently large $n$, there exists a sequence $\{x_{n}\}_{n=1}^{\infty}$ such that $\phi(x_{n})=\phi_{n}$ and $\lim_{n\to \infty} x_{n}=\infty$.
Then we see from \eqref{phieq2} that $(\partial_{x} \phi)^{2}(x_{n})=2V(\phi(x_{n}))=2V(\phi_{n})=0$.
This fact contradicts to $\D_{x} \phi (x)<0$ in Definition \ref{DefS1}.
Therefore, the problem \eqref{sVP1} admits no solution.
\end{proof}

\subsection{Solvability}\label{S3.2}
We prove Assertions (i) and (ii) in Theorem \ref{existence1} on the solvability of the problem \eqref{sVP1}.

\begin{proof}[Proof of Assertions (i) and (ii) in Theorem \ref{existence1}]
We first note that \eqref{V2} holds owing to $\xi_{1}^{-2}f_{\infty} \in L^{1}(\mathbb R^{3})$. 
It is seen that $B \neq \emptyset $ if \eqref{Bohm2} holds,
since the Taylor theorem with \eqref{Bohm2} and \eqref{V2} ensures that $V$ is positive around $\phi=0$.
It is easy to see from \eqref{phieq2} and $\D_{x} \phi <0$ in Definition \ref{DefS1} 
that the condition $\phi_{b} < \sup B$ is necessary for the solvability stated in Assertions (i) and (ii).


Hereafter we discuss simultaneously Assertions (i) and (ii).
Suppose that $\phi_{b} < \sup B$ holds, which implies that $V(\phi)>0$ for $\phi \in (0,\phi_{b}]$.
Let us construct $\phi$ with $\D_{x} \phi<0$ 
by solving \eqref{phieq2} with \eqref{sbc3} and \eqref{sbc4}.
To this end, we rewrite \eqref{phieq2} into the equivalent equation
\begin{gather}\label{phieq3}
\D_{x}{\phi}=-\sqrt{2V({\phi})}.
\end{gather}
The Taylor theorem together with \eqref{V2} and the assumption $\int_{\mathbb R^{3}}\xi_{1}^{-2}f_{\infty}(\xi) d\xi \leq 1$ ensures that 
$\sqrt{V}$ is Lipschitz continuous around $\phi=0$.
Combining this and the fact that $V(\phi)>0$ for $\phi \in (0,\phi_{b}]$,
we deduce that $\sqrt{V}$ is Lipschitz continuous on $[0,\phi_{b}]$.
Therefore, a standard theory of ordinary differential equations gives
the unique solvability of \eqref{phieq3} with \eqref{sbc3} and \eqref{sbc4}.
Then it is straightforward  to see $\D_{x} \phi<0$ and 
$\phi \in C(\overline{\mathbb R_{+}}) \cap C^{2}(\mathbb R_{+})$ 
from the equivalent equation.
Thus we have the desired $\phi$ satisfying \eqref{phieq2} with \eqref{sbc3} and \eqref{sbc4}.
Furthermore, it is easy to show by differentiating \eqref{phieq2} and using $\D_{x} \phi<0$
that $\phi$ satisfies \eqref{phieq1}.

Now we define $f$ as \eqref{fform} by using $\phi$, and prove that $f$ satisfies the conditions (i)--(iii) in Definition \ref{DefS1}.
Owing to \eqref{need2} and $f \in L^{1}(\mathbb R^{3})$, there exists a sequence $\{f_{\infty}^{k}\} \subset C^{\infty}_{0}(\mathbb R^{3})$ 
such that $f_{\infty}^{k}(\xi)=0$ for $\xi_{1}>-1/k$, and $f_{\infty}^{k} \to f_{\infty}$ in $L^{1}(\mathbb R^{3})$ as $k \to \infty$. 

Let us show the condition (i), i.e. $f \in C(\overline{\mathbb R_{+}};L^{1}(\mathbb R^{3}))\cap L^{1}_{loc}(\mathbb R_{+}\times\mathbb R^{3})$. 
First $f(x,\cdot) \in L^{1}(\mathbb R^{3})$ follows from \eqref{rho0} and \eqref{rho0'}.
To investigate the continuity of $f$,
we set $G(x,\xi)=\chi(\xi_{1}^{2}-2\phi(x))\chi(-\xi_{1})$ and observe that for $x,x_{0} \in \overline{\mathbb R_{+}}$,
\begin{align*}
&\|f(x)-f(x_{0})\|_{L^{1}(\mathbb R^{3})}
\\
&=\int_{\mathbb R^{3}}\left| f_{\infty}(-\sqrt{\xi_{1}^{2}-2\phi(x)},\xi')G(x,\xi) -f_{\infty}(-\sqrt{\xi_{1}^{2}-2\phi(x_{0})},\xi')G(x_{0},\xi) \right| d\xi
\\
&\leq \int_{\mathbb R^{3}}\left| f_{\infty}(-\sqrt{\xi_{1}^{2}-2\phi(x)},\xi')G(x,\xi) -f_{\infty}^{k}(-\sqrt{\xi_{1}^{2}-2\phi(x)},\xi')G(x,\xi) \right| d\xi
\\
&\quad +\int_{\mathbb R^{3}}\left| f_{\infty}^{k}(-\sqrt{\xi_{1}^{2}-2\phi(x)},\xi')G(x,\xi) -f_{\infty}^{k}(-\sqrt{\xi_{1}^{2}-2\phi(x_{0})},\xi')G(x_{0},\xi) \right| d\xi
\\
&\quad +\int_{\mathbb R^{3}}\left| f_{\infty}^{k}(-\sqrt{\xi_{1}^{2}-2\phi(x_{0})},\xi')G(x_{0},\xi) -f_{\infty}(-\sqrt{\xi_{1}^{2}-2\phi(x_{0})},\xi') G(x_{0},\xi)\right| d\xi
\\
&=:K_{1}+K_{2}+K_{3}.
\end{align*}
Using the change of variable $\zeta_{1}=-\sqrt{\xi_{1}^{2}-2{\phi}(x)}$ and the fact that $f_{\infty}(\xi)=f_{\infty}^{k}(\xi)=0$ holds for $\xi_{1}>0$, 
we can estimate $K_{1}$ as
\begin{align*}
K_{1} = \int_{\mathbb R^{3}}
| f_{\infty}(\xi) -f_{\infty}^{k}(\xi) |  \frac{|\xi_{1}|}{\sqrt{\xi^{2}_{1}+2{\phi}(x)}}\,d\xi
\leq \|f_{\infty} - f_{\infty}^{k}\|_{L^{1}(\mathbb R^{3})}   \to 0 \quad \text{as $k\to \infty$},
\end{align*}
where we have also used the fact $\phi(x)>0$ in deriving the inequality.
Similarly, $K_{3} \to 0$ as $k \to \infty$. 
Thus $K_{1}$ and $K_{3}$ can be arbitrarily small for suitably large $k$. 
For the fixed $k$, the dominated convergence theorem ensures that $K_{2}$ converges to zero as $x \to x_{0}$.
Hence, we deduce $f \in C(\overline{\mathbb R_{+}};L^{1}(\mathbb R^{3}))$.
Now it is straightforward to show $f \in L^{1}_{loc}(\mathbb R_{+}\times\mathbb R^{3})$.

It is clear that  the condition (ii), i.e. $f \geq 0$, holds.
Let us prove \eqref{weak1} and \eqref{weak2} in the condition (iii). 
Obviously, \eqref{weak2} follows from the same manner as above.
It is also evident that the function $f^{k}$ defined by replacing $f_{\infty}$ by  $f_{\infty}^{k}$ in \eqref{fform} belongs to $C^{\infty}_{0}(\mathbb R^{3})$, and satisfies the weak form \eqref{weak1} for each $k$.
Using  the same change of variable as above and letting $k \to \infty$, 
we see that $f$ also satisfies \eqref{weak1}.
Consequently, all the conditions (i)--(iii) hold.

The condition (iv) is validated by \eqref{fform}, \eqref{rho0'}, and \eqref{phieq1}.
The uniqueness of a solution $(f,\phi)$ follows from 
the uniqueness of the solution of \eqref{phieq3} with \eqref{sbc3} and \eqref{sbc4}.
Using \eqref{need2}, \eqref{fform}, and the fact $\phi \in C(\overline{\mathbb R_{+}}) \cap C^{2}(\mathbb R_{+})$, 
we can show that $f \in C^{1}(\overline{\mathbb R_{+}}\times \mathbb R^{3})$
if $f_{\infty} \in C^{2}(\mathbb R^{3})$. 

The decay estimate \eqref{decay1} follows from \eqref{Bohm2}, \eqref{V2}, and \eqref{phieq3}. The proof is complete.
\end{proof}

For the proof of the delta mass limit as $\ve \to 0$, 
we need to show that  $\sup B$ in Assertion (i) of Theorem \ref{existence1} is bounded from below 
by some positive constant $\delta_{0}$ independent of $\ve$
for the problem \eqref{sVP1} with $f_{\infty}=f_{\infty,\ve}$.
Here $f_{\infty,\ve}$ is defined in \eqref{fve}.

\begin{cor}\label{4.1}
Let $f_{b}=\alpha=0$ and $u_{\infty}>1$.
There exists a positive constant $\delta_{0}$ such that 
if $0<\phi_{b}<\delta_{0}$, then the problem \eqref{sVP1} with $f_{\infty}=f_{\infty,\ve}$
has a unique solution $(f_{\ve},\phi_{\ve}) \in C^{1}(\overline{\mathbb R_{+}}\times \mathbb R^{3}) \times C^{2}(\overline{\mathbb R_{+}})$ for any $\ve \in (0,\ve_{0})$,
where $\ve_{0} := (u_{\infty}-1)/2$.
\end{cor}
\begin{proof}
We first note that $f_{\infty,\ve} \in C_{0}^{\infty}(\mathbb R^{3})$ satisfies
the conditions \eqref{netrual1}, \eqref{need2}, and \eqref{Bohm2} 
being in Theorem \ref{existence1} for any $\ve \in (0,\ve_{0})$.
It suffices to show that there exists $\delta_{0}>0$ independent of $\ve$ such that
\begin{gather}\label{positive1}
V_{\ve}(\phi)>0, \quad \phi \in (0,\delta_{0}],
\end{gather}
where $V_{\ve}$ denotes the function replaced $f_{\infty}$ by $f_{\infty,\ve}$ in \eqref{V}.
Indeed, repeating the proof of Assertions (i) and (ii) of Theorem \ref{existence1} with the aid of \eqref{positive1}, 
we can conclude that the problem \eqref{sVP1} with $f_{\infty}=f_{\infty,\ve}$
has a unique solution $(f_{\ve},\phi_{\ve}) \in C^{1}(\overline{\mathbb R_{+}}\times \mathbb R^{3}) \times C^{2}(\overline{\mathbb R_{+}})$ for $\phi_{b} \in (0,\delta_{0})$ and
$\ve \in (0,\ve_{0})$. 

Let us complete the proof by showing \eqref{positive1}. 
Using $\int_{\mathbb R^{3}} f_{\infty,\ve}(\xi) d\xi=1$, we observe that 
\begin{align*}
\left|\frac{d^{3} V_{\ve}}{d\phi^{3}}(\phi)\right|
&= \left|3\int_{\mathbb R^{3}} f_{\infty,\ve}(\xi)\frac{-\xi_{1}}{(\xi_{1}^{2}+2\phi)^{5/2}} d\xi - n_{e}''(\phi) \right|
\leq 3\int_{\mathbb R^{3}} f_{\infty,\ve}(\xi)|\xi_{1}|^{-4}d\xi + C_{0}
\\
&\leq 3(\sup_{|\xi_{1}+u_{\infty}| \leq \ve_{0}} |\xi_{1}|^{-4}) \int_{\mathbb R^{3}} f_{\infty,\ve}(\xi) d\xi +C_{0}
\leq C_{0} 
\end{align*}
for any $\phi \in (0,1]$, where $C_{0}$ is independent of $\ve$. 
Recalling \eqref{V2}, we also see that
\begin{gather*}
\frac{d^{2} V_{\ve}}{d\phi^{2}}(0)
=-\int_{\mathbb R^{3}} f_{\infty,\ve}(\xi)|\xi_{1}|^{-2}d\xi +1
\geq  -(\sup_{|\xi_{1}+u_{\infty}| \leq \ve_{0}} |\xi_{1}|^{-2}) +1  >0,
\end{gather*}
where we have used $\ve_{0} = (u_{\infty}-1)/2$ and $u_{\infty}^{2}>1$.
By the Taylor theorem and the above two inequalities,
we conclude that there exists $\delta_{0}$ independent of $\ve$ such that
\begin{gather*}
\frac{d^{2} V_{\ve}}{d\phi^{2}}(\phi) 
\geq -(\sup_{|\xi_{1}+u_{\infty}| \leq \ve_{0}} |\xi_{1}|^{-2}) +1 -C_{0}\delta_{0} >0,
\quad \phi \in [0,\delta_{0}],
\end{gather*}
which together with \eqref{V2} leads to \eqref{positive1}.
The proof is complete.
\end{proof}

\subsection{The Delta Mass Limit}\label{S3.3}

This section deals with the proof of Theorem \ref{DiffThm1} on the delta mass limit.

\begin{proof}[Proof of Theorem \ref{DiffThm1}]
Let $(f_{\ve},\phi_{\ve})$ be a solution in Corollary \ref{4.1}, 
and also $(\rho_{\ve},u_{\ve})$ be the moments of $f_{\ve}$.
We first show the estimate of $\phi_{\ve}-\phi$ in \eqref{conver0}. 
From \eqref{sp3}, \eqref{EPform1}, \eqref{rho0'}, and \eqref{phieq1},
we observe that
\begin{align}
\partial_{xx} (\phi_{\ve}-\phi)
= W(\phi_{\ve}) - W(\phi) + R_{\ve}(x), \quad x>0,
\label{DiffEq0}
\end{align}
where
\begin{gather*}
W(\phi):= \frac{u_{\infty}}{\sqrt{u_{\infty}^{2}+2\phi}} - e^{-\phi},
\quad R_{\ve}(x):=\int_{\mathbb R^{3}} f_{\infty,\ve}(\xi)\frac{-\xi_{1}}{\sqrt{\xi_{1}^{2}+2\phi_{\ve}(x)}}  d\xi 
- \frac{u_{\infty}}{\sqrt{u_{\infty}^{2}+2\phi_{\ve}(x)}}.
\end{gather*}
Set $M_{\ve}:=\sup_{x \in {\mathbb R}_{+}} \left| R_{\ve}(x)\right|$.
Owing to the assumption $u_{\infty}>1$, we have ${d W }/{d\phi } (0)>0$.
Then it is seen by taking $\delta_{0}>0$ and $c_{0}>0$ small enough that 
\begin{gather}\label{W1}
c_{0} \leq  \frac{d W }{d\phi } (\phi) \leq c_{0}^{-1}, \quad \phi \in [0,\delta_{0}].
\end{gather}

Now we claim that
\begin{gather}\label{DiffEq3}
\sup_{x \in {\mathbb R}_{+}}|(\phi_{\ve}-\phi)(x)| \leq  M_{\ve}/c_{0}.
\end{gather}
To show this, we observe from \eqref{DiffEq0} that 
\begin{align}
-\partial_{xx} (\phi_{\ve}-\phi-M_{\ve}/c_{0})
+\left(W(\phi_{\ve}) - W(\phi) -M_{\ve} \right) \leq 0.
\label{DiffEq1}
\end{align}
Furthermore, $(\phi_{\ve}-\phi-M_{\ve}/c_{0})_{+}:=\max\{0, \ \phi_{\ve}-\phi-M_{\ve}/c_{0}\} \in H^{1}({\mathbb R}_{+})$ follows from \eqref{sbc4} and \eqref{sp4}.
Multiply \eqref{DiffEq1} by $(\phi_{\ve}-\phi-M_{\ve}/c_{0})_{+}$,
integrate it over $(0,\infty)$, and use $(\phi_{\ve}-\phi-M_{\ve}/c_{0})_{+} (0)=0$ to obtain
\begin{align}
\|\D_{x}(\phi_{\ve}-\phi-M_{\ve}/c_{0})_{+} \|_{L^{2}(\mathbb R_{+})}^{2}
+ \int_{0}^{\infty} \left(W(\phi_{\ve}) - W(\phi) -M_{\ve}\right)
(\phi_{\ve}-\phi-M_{\ve}/c_{0})_{+}dx \leq 0.
\label{DiffEq2}
\end{align}
The second term on the left hand side is positive. 
Indeed, using the mean value theorem, \eqref{W1}, and $0 \leq \phi,\phi_{\ve} \leq \delta_{0}$,
we arrive at
\begin{align*}
& \int_{0}^{\infty} \left(W(\phi_{\ve}) - W(\phi) -M_{\ve}\right)(\phi_{\ve}-\phi-M_{\ve}/c_{0})_{+}dx
\notag \\
& = \int_{0}^{\infty} \left(\int_{0}^{1}\frac{dW}{d\phi}(\theta \phi_{\ve}+(1-\theta)\phi)d\theta(\phi_{\ve}-\phi) -M_{\ve}\right)(\phi_{\ve}-\phi-M_{\ve}/c_{0})_{+}dx
\geq 0.
\end{align*}
Therefore, \eqref{DiffEq2} gives $(\phi_{\ve}-\phi-M_{\ve}/c_{0})_{+} = 0$ 
which means $\phi_{\ve}-\phi\leq M_{\ve}/c_{0}$. 
Similarly, $\phi_{\ve}-\phi \geq - M_{\ve}/c_{0}$ holds.
Thus the claim \eqref{DiffEq3} is vaild.

Next let us estimate $M_{\ve}$ by $C_{0} \ve$ as 
\begin{align*}
M_{\ve}
&= \sup_{x \in {\mathbb R}_{+}} \left|\int_{\mathbb R^{3}} f_{\infty,\ve}(\xi) \left(
\frac{-\xi_{1}}{\sqrt{\xi_{1}^{2}+2\phi_{\ve}(x)}}-\frac{u_{\infty}}{\sqrt{u_{\infty}^{2}+2\phi_{\ve}(x)}}  \right)d\xi \right| 
\notag \\
& \leq  \sup_{x \in {\mathbb R}_{+}} \left(\sup_{|\xi-(-u_{\infty},0,0)| \leq \ve} 
\left|\frac{-\xi_{1}}{\sqrt{\xi_{1}^{2}+2\phi_{\ve}(x)}}-\frac{u_{\infty}}{\sqrt{u_{\infty}^{2}+2\phi_{\ve}(x)}} \right| \right)
\int_{\mathbb R^{3}} f_{\infty,\ve}(\xi)d\xi
\notag \\
& \leq C_{0} \ve,
\end{align*}
where we have used the fact $\int_{\mathbb R^{3}} f_{\infty,\ve}(\xi)d\xi=1$ in deriving the equality,
and $C_{0}$ is a positive constant independent of $\ve$.
This together with \eqref{DiffEq3} leads to the estimate of $\phi_{\ve}-\phi$ in \eqref{conver0}.

We complete the proof by showing the other estimates in \eqref{conver0}.
Let us first handle $\rho_{\ve}-\rho$.
It is clear that $\rho_{\ve}(x)=\rho_{i}(\phi_{\ve}(x))$ holds owing to \eqref{rho0'}.
With the aid of this and \eqref{EPform1}, we estimate $\rho_{\ve}-\rho$ as
\begin{align*}
\sup_{x \in {\mathbb R}_{+}} \left| \rho_{\ve}(x)-\rho(x) \right|
&=\sup_{x \in {\mathbb R}_{+}} \left| \int_{\mathbb R^{3}} f_{\infty,\ve}(\xi) 
\left(\frac{-\xi_{1}}{\sqrt{\xi_{1}^{2}+2\phi_{\ve}(x)}} 
- \frac{u_{\infty}}{\sqrt{u_{\infty}^{2}+2\phi(x)}} \right) d\xi\right|
\notag \\
& \leq \sup_{x \in {\mathbb R}_{+}} \left(\sup_{|\xi-(-u_{\infty},0,0)| \leq \ve} 
\left|\frac{-\xi_{1}}{\sqrt{\xi_{1}^{2}+2\phi_{\ve}(x)}}-\frac{u_{\infty}}{\sqrt{u_{\infty}^{2}+2\phi(x)}} \right| \right)
\int_{\mathbb R^{3}} f_{\infty,\ve}(\xi)d\xi
\notag \\
& \leq C_{0} (\ve+\sup_{x \in {\mathbb R}_{+}}|\phi(x)-\phi_{\ve}(x)|) \leq C_{0} \ve.
\end{align*}
Let us treat $\rho_{\ve}u_{\ve}-\rho u$. 
By the change of variable $\sqrt{\xi_{1}^{2}-2\phi(x)}=-\zeta_{1}$, the term $\rho_{\ve}u_{\ve}$ is rewritten as
\begin{gather*}
(\rho_{\ve}u_{\ve})(x)=\int_{\mathbb R^{3}} \xi_{1} f_{\infty,\ve}(\xi)  d\xi.
\end{gather*}
It is easy to see $\rho u=-u_{\infty}$ from \eqref{EPform1}.
Then we arrive at
\begin{align*}
\sup_{x \in {\mathbb R}_{+}} \left| (\ro_{\ve}u_{\ve}) (x) - (\rho u)(x) \right|
&=\sup_{x \in {\mathbb R}_{+}} \left| \int_{\mathbb R^{3}} f_{\infty,\ve}(\xi) 
(\xi_{1}+u_{\infty}) d\xi\right| \leq  \ve.
\end{align*}
Consequently, \eqref{conver0} holds.
\end{proof}

\section{General Boundaries}\lb{S4}

In this section, we prove Theorems \ref{existence2}--\ref{DiffThm2} 
on the solvability and delta mass limit for general boundaries. 
Subsections \ref{S4.1} and \ref{S4.2} deal with the solvability for
the attractive and repulsive boundaries, respectively.
In subsection \ref{S4.3}, we study the delta mass limit.

\subsection{Solvability for the Attractive Boundary}\lb{S4.1}
We start from studying the necessary conditions for the solvability of the problem \eqref{sVP1}
similarly as in Section \ref{S3}.
Specifically, we show the following lemma which ensures immediately Assertion (iv) in Theorem \ref{existence2}.
\begin{lem}\label{lem41}
Let $\phi_{b}>0$, $f_b \in L^1(\R_+^3)$, $\alpha\neq1$, $f_{\infty} \in L^{1}(\mathbb R^{3})$, $f_{b}\geq0$, $f_{\infty} \geq 0$, and \eqref{netrual1} hold.
Suppose that  the problem \eqref{sVP1} has a solution $(f,\phi)$.
Then $\rho^{+} \in C([0,\phi_{b}])$, $V^{+} \in C^{1}([0,\phi_{b}])$, and the function $f_{\infty}$ satisfies the condition \eqref{need3}.
Furthermore, $\phi$ solves \eqref{phieq5} below with \eqref{sbc3} and \eqref{sbc4},
and $f$ is written as \eqref{fform2} by $\phi$.
If $V^{+} \in C^{2}([0,\phi_{b}])$,  the generalized Bohm criterion holds:
\begin{gather}\label{Bohm5}
\frac{d^{2} V^{+}}{d \phi^{2}}(0) \geq 0.
\end{gather}
\end{lem} 
\begin{proof}
In much the same way as in the proof of Lemma \ref{lem31}, we see 
$\D_{x}\phi(x)<0$ and  $\lim_{x \to \infty}\D_{x}\phi(x)$ $=0$.
Let us first show that $f_{\infty}$ satisfies \eqref{need3} and $f$ is written by \eqref{fform2}.
Applying the characteristics method to \eqref{weak1}, 
we see that the value of $f$ must remain a constant along the characteristic curves
$\xi_{1}^{2}/2-\phi(x)=c$,
where $c$ is some constant. 
The characteristics are drawn as in Figure \ref{fig1}.
Therefore, there hold that
\begin{align*}
{f}(y,-\sqrt{\eta_1^2+2{\phi}(y)},\eta')=f_\infty(\eta), 
& \quad
(y,\eta_1,\eta')\in X_{1},
\\
f(y,\pm\sqrt{2{\phi}(y)-\eta_1^2},\eta')
=\frac{1}{1-\alpha}\,f_b(\sqrt{2\phi_b-\eta_1^2},\eta'),
& \quad
(y,\eta_1,\eta')\in  X_{2},
\\
f(y,\sqrt{\eta_1^2+2{\phi}(y)},\eta')
=f_\infty(\eta) 
=f_b(\sqrt{\eta_1^2+2\phi_b},\eta')
+\alpha f_\infty(-\eta_1,\eta'), 
& \quad
(y,\eta_1,\eta')\in X_{3},
\end{align*}
where $\eta=(\eta_{1},\eta_{2},\eta_{3})=(\eta_{1},\eta')$ and
\[
X_{1}:=\overline{\R_+}\times\R_-^3, \quad
X_{2}:=\{(y,\eta_1,\eta')\in\overline{\R_+}\times\R^3\;|\;\eta_1^2<2{\phi}(y)\leq 2\phi_b\}, \quad
X_{3}:=\overline{\R_+}\times\R_+^3. 
\]
The last equality means that \eqref{need3} must hold.
Furthermore, we conclude from these three equalities  \footnote{If $\alpha=1$ and $f_{b}=0$, we have multiple choices for the second term on the right hand side of \eqref{fform2}.}that $f$ must be written by \eqref{fform2}, i.e.
\begin{gather*}
\begin{aligned}
{f}(x,\xi)
& =f_\infty(-\sqrt{\xi_1^2-2{\phi}(x)},\xi')
\chi(\xi_{1}^{2}-2\phi(x))\chi(-\xi_{1}) \\
& \qquad
+\frac{1}{1-\alpha}\,f_b(\sqrt{\xi_1^2-2{\phi}(x)+2\phi_b},\xi')
\chi(-\xi_1^2+2{\phi}(x)) \\
& \qquad
+f_\infty(\sqrt{\xi_1^2-2{\phi}(x)},\xi')
\chi(\xi_{1}^{2}-2\phi(x))\chi(\xi_{1}).
\end{aligned}
\end{gather*}

Next we show that $\phi$ satisfies \eqref{phieq5} below.
Integrating \eqref{fform2} over $\mathbb R^{3}$ and using 
the change of variables $\sqrt{\xi_{1}^{2}-2\phi}=-\zeta_{1}$, $\sqrt{\xi_1^2-2\phi+2\phi_b}=\zeta_1$, and $\sqrt{\xi_{1}^{2}-2\phi}=\zeta_{1}$
for the first, second, and third terms on the right hand side of \eqref{fform2}, respectively, we see that
\begin{align}
\int_{\mathbb R^{3}} f(x,\xi) \,d\xi
&=\int_{{\mathbb R}^3}f_\infty(\xi)\frac{|\xi_1|}{\sqrt{\xi_1^2+2\phi(x)}}\,d\xi
\notag\\
&\quad +\frac{2}{1-\alpha}\int_{\sqrt{2(\phi_b-\phi(x))}}^{\sqrt{2\phi_b}}\int_{\mathbb R^2}
f_b(\xi)\frac{\xi_1}{\sqrt{\xi_1^2+2\phi(x)-2\phi_b}}\,d\xi_1\,d\xi'
\notag\\
&=\rho_{i}^{+}(\phi(x))
\label{rho+'},
\end{align}
where $\phi(x) \in [0,\phi_{b}]$ and $\rho_{i}^{+}$ is the same function defined in \eqref{rho++}.
Substituting \eqref{rho+'} into \er{seq2} yields an ordinary differential equation for $\phi$:
\begin{equation}\lb{phieq5+} 
\partial_{xx} \phi 
= \rho_{i}^{+}(\phi) 
- n_{e}(\phi), \quad x>0.
\end{equation}
Multiply \er{phieq5+} by $\D_{x}\phi$, integrate it over $(x,\infty)$, and use \eqref{sbc4} and $\lim_{x \to \infty}\D_{x}\phi(x)=0$ to obtain
\begin{equation}\lb{phieq5}
(\D_{x} \phi)^2=2V^{+}(\phi), \quad V^{+}(\phi)=\int_{0}^{\phi} \left(\rho_{i}^{+}(\varphi) - n_{e}(\varphi) \right)  d\vphi,
\end{equation}
where the function $V^{+}$ is the same function defined in \eqref{V+}.
Thus $\phi$ must satisfy \eqref{phieq5}.

Now we show $\rho_{i}^{+} \in C([0,\phi_{b}])$ which immediately gives $V^{+} \in C^{1}([0,\phi_{b}])$.
Owing to $ f \in C(\overline{\R_{+}};L^{1}(\R^{3}))$ in Definition \ref{DefS1} and \eqref{weak2}, there hold that
\begin{align*}
\rho_{i}^{+}(\phi(x))= \| f(x) \|_{L^{1}} \in C(\overline{\R_{+}}),
\quad
\lim_{x\to\infty}\| f(x) \|_{L^{1}} = \| f_{\infty} \|_{L^{1}}=1.
\end{align*}
These imply  $\rho_{i}^{+} \in C([0,\phi_{b}])$ with the aid of \eqref{sbc4},
$ \phi \in C(\overline{\R_{+}})$, and $\D_{x} \phi <0$ in Definition \ref{DefS1}.

What is left is to obtain the generalized Bohm criterion \eqref{Bohm5}.
We see from $n_{e}(0)=1$ and \eqref{netrual1} that
\begin{equation}\lb{V5}
V^{+}(0)=\frac{d V^{+}}{d \phi}(0)=0.
\end{equation}
If $V^{+} \in C^{2}([0,\phi_{b}])$, we arrive at \eqref{Bohm5}
from \eqref{phieq5} and \eqref{V5} with the aid of the Taylor theorem. 
The proof is complete.
\end{proof}

Assertion (iii) in Theorem \ref{existence2} can be shown 
by the same method as in the proof of Assertion (iii) in Theorem \ref{existence1} with the aid of \eqref{phieq5}.
We omit the proof.
Let us show Assertions (i) and (ii) in Theorem \ref{existence2}.

\begin{proof}[Proof of Assertions (i) and (ii) in Theorem \ref{existence2}]
We first note that \eqref{V5} holds. It is seen that $B^{+}\neq \emptyset $ if $d^{2} V^{+}/d\phi^{2}(0)>0$,
since the Taylor theorem with \eqref{V5} ensures that $V^{+}$ is positive around $\phi=0$.
Let us show that $\phi_{b} < \sup B^{+}$ is a necessary condition for the solvability stated in Assertions (i) and (ii).
Suppose that a solution $(f,\phi)$ exists for $\phi_{b} \geq \sup B^{+}$.
Then $\sup B^{+} \in (0,\phi_{b}]$ and $V^{+}(\sup B^{+})=0$.
Furthermore, there exists a unique $x_{*} \in \overline{\mathbb R_{+}}$ so that $\phi(x_{*})=\sup B^{+}$.
On the other hand, due to \eqref{phieq5}, it is seen that $(\partial_{x} \phi(x_{*}))^{2}=2V^{+}(\sup B^{+})=0$. This contradicts to $\partial_{x} \phi (x)<0$ in Definition \ref{DefS1}. 
Hence, $\phi_{b} < \sup B^{+}$ is necessary.

Hereafter we discuss simultaneously Assertions (i) and (ii).
Suppose that $\phi_{b} < \sup B^{+}$ holds, which implies that $V^{+}(\phi)>0$ for $\phi \in (0,\phi_{b}]$.
The Taylor theorem with \eqref{V5} and the assumption $d^{2} V^{+}/d\phi^{2}(0)\geq 0$ ensures that 
$\sqrt{V^{+}}$ is Lipschitz continuous around $\phi=0$.
Combining this and the fact that $V^{+}(\phi)>0$ for $\phi \in (0,\phi_{b}]$,
we deduce that $\sqrt{V^{+}}$ is Lipschitz continuous on $[0,\phi_{b}]$.
Hence, we can define $\phi$ by solving \eqref{phieq5} with \eqref{sbc3} and \eqref{sbc4}.
Then it is seen that $\phi$ satisfies $\partial_{x}\phi<0$, $\phi \in C(\overline{\mathbb R_{+}}) \cap C^{2}(\mathbb R_{+})$, and \eqref{phieq5+}  by following the proof of  Assertions (i) and (ii) of Theorem \ref{existence1} in subsection \ref{S3.2}.

Now by using $\phi$, we define $f$ as \eqref{fform2}, i.e. 
\begin{gather*}
\begin{aligned}
{f}(x,\xi)
& =f_\infty(-\sqrt{\xi_1^2-2{\phi}(x)},\xi')
\chi(\xi_{1}^{2}-2\phi(x))\chi(-\xi_{1}) \\
& \qquad
+\frac{1}{1-\alpha}\,f_b(\sqrt{\xi_1^2-2{\phi}(x)+2\phi_b},\xi')
\chi(-\xi_1^2+2{\phi}(x)) \\
& \qquad
+f_\infty(\sqrt{\xi_1^2-2{\phi}(x)},\xi')
\chi(\xi_{1}^{2}-2\phi(x))\chi(\xi_{1})
\\
&=:F_{1}(x,\xi)+F_{2}(x,\xi)+F_{3}(x,\xi),
\end{aligned}
\end{gather*}
and prove that $f$ satisfies the conditions (i)--(iii) in Definition \ref{DefS1}.
Owing to \eqref{need3}, $f_{b} \in L^{1}(\mathbb R^{3}_{+})\cap L^{r}(0,\sqrt{2\phi_{b}};L^{1}(\mathbb R^{2}))$, and $f_{\infty} \in L^{1}(\mathbb R^{3})$,
there exist sequences $\{f_{\infty}^{k}\}, \{f_{b}^{k}\} \subset C^{\infty}_{0}(\mathbb R^{3})$ 
such that \eqref{need3} with $(f_{\infty},f_{b})=(f_{\infty}^{k},f_{b}^{k}) $ holds;
$f_{\infty}^{k}(\xi)=0$ holds for $|\xi_{1}|<1/k$; $f_{b}^{k}(\xi)=0$ holds for $|\xi_{1}-\sqrt{2\phi_{b}}|<1/k$;
$f_{\infty}^{k} \to f_{\infty}$  in $L^{1}(\mathbb R^{3})$ as $k \to \infty$;
$f_{b}^{k} \to f_{b}$  in $L^{1}(\mathbb R^{3}_{+}) \cap L^{r}(0,\sqrt{2\phi_{b}};L^{1}(\mathbb R^{2}))$ as $k \to \infty$.
Let us show the condition (i), i.e. $f \in C(\overline{\mathbb R_{+}};L^{1}(\mathbb R^{3}))\cap L^{1}_{loc}(\mathbb R_{+}\times\mathbb R^{3})$.
It is seen that $F_{1}, F_{3} \in C(\overline{\mathbb R_{+}};L^{1}(\mathbb R^{3}))$ 
in much the same way as the proof of Assertions (i) and (ii) of Theorem \ref{existence1}. 
The fact $F_{2}(x,\cdot) \in L^{1}(\mathbb R^{3})$ follows from \eqref{rho+} and \eqref{rho+'}.
To investigate the continuity of $F_{2}$, we set $G(x,\xi)=\chi(-\xi_1^2+2{\phi}(x))$ and observe that for $x,x_{0} \in \overline{\mathbb R_{+}}$,
\begin{align*}
&(1-\alpha)\|F_{2}(x)-F_{2}(x_{0})\|_{L^{1}(\mathbb R^{3})}
\\
&=\int_{\mathbb R^{3}}\left| f_b(\sqrt{\xi_1^2-2{\phi}(x)+2\phi_b},\xi')G(x,\xi)-f_b(\sqrt{\xi_1^2-2{\phi}(x_{0})+2\phi_b},\xi')G(x_{0},\xi) \right| d\xi
\\
&\leq \int_{\mathbb R^{3}}\left| f_b(\sqrt{\xi_1^2-2{\phi}(x)+2\phi_b},\xi')G(x,\xi)-f_b^{k}(\sqrt{\xi_1^2-2{\phi}(x)+2\phi_b},\xi')G(x,\xi)\right| d\xi
\\
&\quad +\int_{\mathbb R^{3}}\left| f_b^{k}(\sqrt{\xi_1^2-2{\phi}(x)+2\phi_b},\xi')G(x,\xi)-f_b^{k}(\sqrt{\xi_1^2-2{\phi}(x_{0})+2\phi_b},\xi')G(x_{0},\xi) \right| d\xi
\\
&\quad +\int_{\mathbb R^{3}}\left| f_b(\sqrt{\xi_1^2-2{\phi}(x_{0})+2\phi_b},\xi')G(x_{0},\xi)-f_b(\sqrt{\xi_1^2-2{\phi}(x_{0})+2\phi_b},\xi')G(x_{0},\xi) \right| d\xi
\\
&=:K_{1}+K_{2}+K_{3}.
\end{align*}
Using the change of variable $\zeta_{1}=\sqrt{\xi_1^2-2{\phi}(x)+2\phi_b}$, we can estimate $K_{1}$ as
\begin{align*}
K_{1}&= 2\int_{\sqrt{2(\phi_b-{\phi}(x))}}^{\sqrt{2\phi_b}} \frac{\xi_1}{\sqrt{\xi_1^2+2\phi(x)-2\phi_b}} 
\left( \int_{\mathbb R^2} |f_b(\xi)-f^{k}_b(\xi)|  d\xi' \right) d\xi_1  
\\
& \leq 2 \sup_{x \in \mathbb R_{+}}\left(\int_{\sqrt{2(\phi_b-{\phi}(x))}}^{\sqrt{2\phi_b}} \frac{\xi_1^{r'}}{(\xi_1^2+2\phi(x)-2\phi_b)^{r'/2}} \,d\xi_1\right)^{1/r'} \| f_b-f^{k}_b \|_{L^{r}(0,\sqrt{2\phi_{b}};L^{1}(\mathbb R^{2}))} 
\\
&\leq  C \| f_b-f^{k}_b\|_{L^{r}(0,\sqrt{2\phi_{b}};L^{1}(\mathbb R^{2}))}    \to 0 \quad \text{as $k\to \infty$},
\end{align*}
where $r'<2$ is the H\"older conjugate of $r>2$, and $C$ is some positive constant. 
Similarly, $K_{3} \to 0$ as $k \to \infty$. 
Thus $K_{1}$ and $K_{3}$ can be arbitrarily small for suitably large $k$. 
For the fixed $k$, the dominated convergence theorem ensures that $K_{2}$ converges to zero as $x \to x_{0}$.
Hence, we deduce that $F_{2} \in C(\overline{\mathbb R_{+}};L^{1}(\mathbb R^{3}))$ holds and so does 
$f \in C(\overline{\mathbb R_{+}};L^{1}(\mathbb R^{3}))$.
Now it is straightforward to show $f \in L^{1}_{loc}(\mathbb R_{+}\times\mathbb R^{3})$.

It is clear that  the condition (ii), i.e. $f \geq 0$, holds.
Let us prove \eqref{weak1} and \eqref{weak2} in the condition (iii). 
Obviously, \eqref{weak2} follows from the same manner as above.
It is also evident that the function $f^{k}$ defined by replacing $(f_{\infty},f_{b})$ by  $(f_{\infty}^{k},f_{b}^{k})$ in \eqref{fform2} belongs to $C^{\infty}_{0}(\mathbb R^{3})$, and satisfies the weak form \eqref{weak1} for each $k$.
Letting $k \to \infty$, we see that $f$ also satisfies \eqref{weak1}.
Consequently, all the conditions (i)--(iii) hold.

The condition (iv) is validated by \eqref{fform2}, \eqref{rho+'}, and \eqref{phieq5+}.
The uniqueness of a solution $(f,\phi)$ follows from 
the uniqueness of the solution of \eqref{phieq5} with \eqref{sbc3} and \eqref{sbc4}.
We have the decay estimate \eqref{decay2} from \eqref{phieq5}, \eqref{V5}, and
${d^{2} V^{+}}/{d \phi^{2}}(0) >0$. The proof is complete.
\end{proof}

We also have Corollary \ref{4.2} stating
that $\sup B^{+}$ in Assertion (i) of Theorem \ref{existence2} is bounded from below 
by some positive constant $\delta_{0}$ independent of $\ve$ and $\phi_{b}$
for the problem \eqref{sVP1} with $(f_{b},f_{\infty})=(g_{b,\ve},g_{\infty,\ve})$,
where $g_{b,\ve}$ and $g_{\infty,\ve}$ are defined in \eqref{fb2} and \eqref{fve2}.
As mentioned just after \eqref{velocity1}, the functions $g_{b,\ve}$ and $g_{\infty,\ve}$ satisfy
the conditions \eqref{netrual1}, \eqref{need3}, and $d^{2} V^{+}/d\phi^{2}(0)>0$ being in Theorem \ref{existence2} for $|\phi_{b}| \ll 1$ and $\ve\ll 1$ provided that \eqref{velocity1} holds. Furthermore, $B^{+}$ is independent of $\phi_{b}$, and also $\phi_{b}<\sup B^{+}$ holds for $\phi_{b} \ll 1$.
We omit the proof of Corollary \ref{4.2}, since it is the same as the proof of Corollary \ref{4.1}.

\begin{cor}\label{4.2}
Let $\alpha \neq 1$ and \eqref{velocity1} hold.
There exist positive constants $\ve_{0}$ and $\delta_{0}$ such that 
if $0<\phi_{b}<\delta_{0}$, then the problem \eqref{sVP1} with $(f_{b},f_{\infty})=(g_{b,\ve},g_{\infty,\ve})$
has a unique solution $(f_{\ve},\phi_{\ve}) \in C^{1}(\overline{\mathbb R_{+}}\times \mathbb R^{3}) \times C^{2}(\overline{\mathbb R_{+}})$ for any $\ve \in (0,\ve_{0})$.
\end{cor}

\subsection{Solvability for the Repulsive Boundary}\lb{S4.2}
We first study the necessary conditions for the solvability of the problem \eqref{sVP1}
similarly as in subsection \ref{S4.1}. 
Specifically, we show the following lemma which ensures immediately Assertion (iv) in Theorem \ref{existence3}.

\begin{lem}\label{lem42}
Let $\phi_{b}<0$, $f_b \in L^1(\R_+^3)$, $f_{\infty} \in L^{1}(\mathbb R^{3})$, $f_{b}\geq0$, $f_{\infty} \geq 0$, and \eqref{netrual1} hold.
Suppose that  the problem \eqref{sVP1} has a solution $(f,\phi)$.
Then $\rho^{-} \in C([\phi_{b},0])$, $V^{-} \in C^{1}([\phi_{b},0])$, and the function $f_{\infty}$ satisfies the condition \eqref{need4}.
Furthermore, $\phi$ solves \eqref{phieq6} below with \eqref{sbc3} and \eqref{sbc4},
and $f$ is written as \eqref{fform5} by $\phi$.
If $V^{-} \in C^{2}([\phi_{b},0])$,  the generalized Bohm criterion also holds:
\begin{gather}\label{Bohm6}
\frac{d^{2} V^{-}}{d \phi^{2}}(0) \geq 0.
\end{gather}
\end{lem} 
\begin{proof}
In the same manner as in subsection \ref{S3.1}, we see 
$\D_{x}\phi(x)>0$ and  $\lim_{x \to \infty}\D_{x}\phi(x)=0$.
Let us first show that $f_{\infty}$ satisfies \eqref{need4} and $f$ is written by \eqref{fform5}.
Applying the characteristics method to \eqref{weak1}, 
we see that the value of $f$ must remain a constant along the characteristic curves
$\xi_{1}^{2}/2-\phi(x)=c$,
where $c$ is some constant. 
We draw the illustration of characteristics in Figure \ref{fig2} below.
\begin{figure}[htbp]
\begin{center}
   \includegraphics[width=0.5\textwidth]{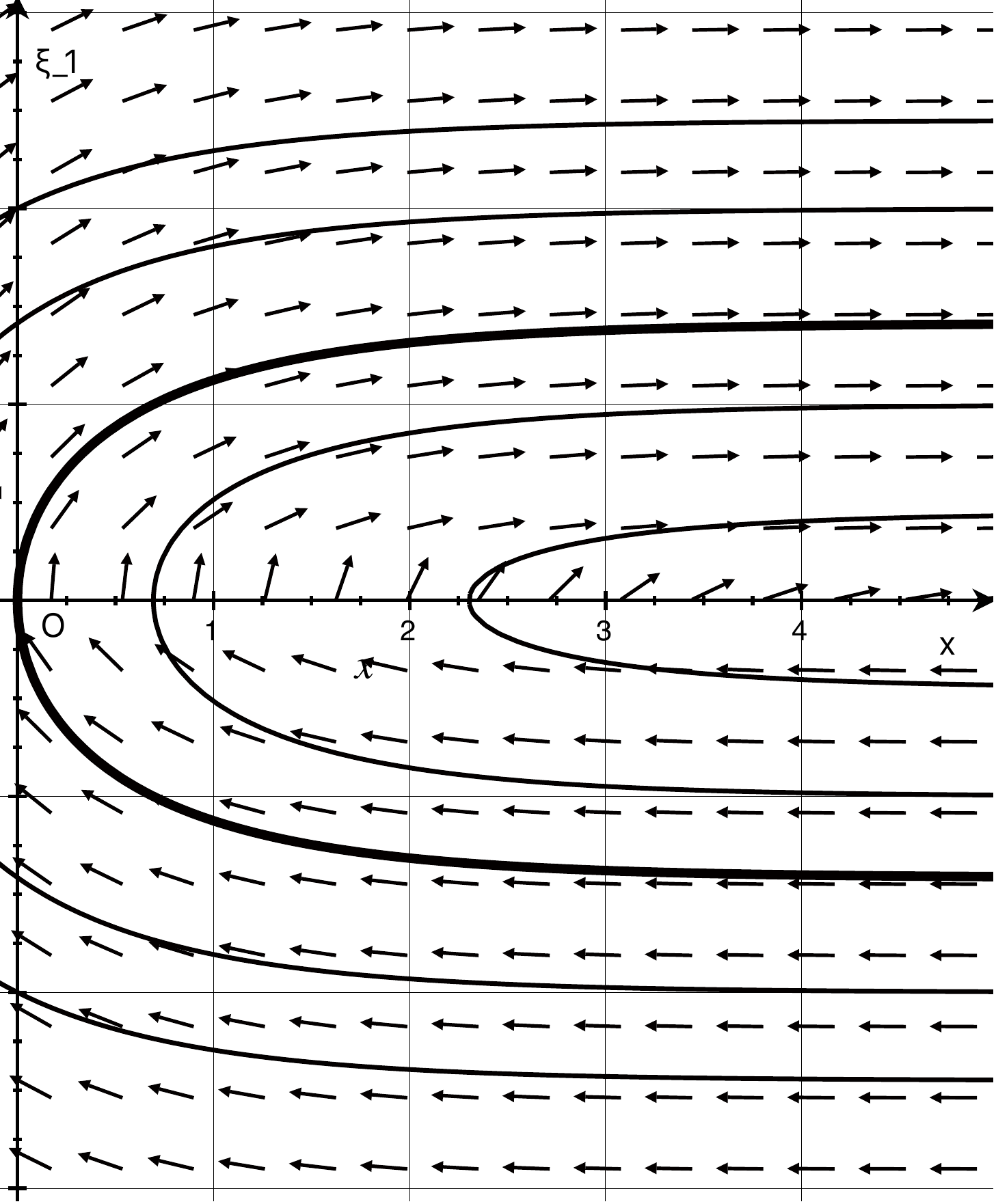}
\end{center}
  \caption{characteristics curves for the case $\partial_{x} \phi >0$}  
  \label{fig2}
\end{figure}
It tells us that
\begin{align*}
f(y,-\sqrt{\eta_{1}^{2}+2\phi(y)},\eta')=f_{\infty}(\eta), &\quad (y,\eta_{1},\eta') \in Y_{1},
 \\
{f}(y,\pm\sqrt{\eta_1^2+2{\phi}(y)},\eta')=f_\infty(\eta), &\quad (y,\eta_{1},\eta') \in Y_{2},
 \\
{f}(y,\sqrt{\eta_1^2+2{\phi}(y)},\eta')=f_\infty(\eta)=f_b(\sqrt{\eta_1^2+2\phi_b},\eta')
+\alpha f_\infty(-\eta_1,\eta'), & \quad  (y,\eta_{1},\eta') \in Y_{3},
\end{align*}
where 
\begin{align*}
Y_{1}&:=\overline{\R_+}\times(-\infty,-\sqrt{2|\phi_b|})\times\R^2,
\\
Y_{2}&:=\{(y,\eta_1,\eta')\in\overline{\R_+}\times\R^3\;|\;
2|{\phi}(y)|<\eta_1^2<2|\phi_b|\}, \quad
\\
Y_{3}&:=\overline{\R_+}\times(\sqrt{2|\phi_b|},\infty)\times\R^2. 
\end{align*}
The second and third equalities mean that \eqref{need4} must hold.
Furthermore, we conclude from these three equalities that $f$ must be written by \eqref{fform5}, i.e, 
\begin{align*}
{f}(x,\xi)
=f_\infty(-\sqrt{\xi_1^2-2{\phi}(x)},\xi')\chi(-\xi_1)
+f_\infty(\sqrt{\xi_1^2-2{\phi}(x)},\xi')\chi(\xi_1). 
\end{align*}

Next we show that $\phi$ satisfies \eqref{phieq6} below.
Integrating \eqref{fform5} over $\mathbb R^{3}$ and
using the change of variables $\sqrt{\xi_{1}^{2}-2\phi}=-\zeta_{1}$ and $\sqrt{\xi_{1}^{2}-2\phi}=\zeta_{1}$
for the first and second terms of the right hand side of \eqref{fform5}, respectively,
we see that 
\begin{align}
\int_{\mathbb R^{3}} f(x,\xi) \,d\xi
&=\int_{{\mathbb R}^3}f_\infty(\xi)\frac{|\xi_1|}{\sqrt{\xi_1^2+2\phi(x)}} \chi(\xi_1^2+2\phi(x))\,d\xi 
=\rho_{i}^{-}(\phi(x)),
\label{rho-'}
\end{align}
where $\rho_{i}^{-}$ is the same function defined in \eqref{rho--}.
Substituting \eqref{rho-'} into \er{seq2} yields an ordinary differential equation for $\phi$:
\begin{equation}\lb{phieq6+} 
\partial_{xx} \phi = \rho_{i}^{-}(\phi) - n_{e}(\phi), \quad x>0.
\end{equation}
Multiply \er{phieq6+} by $\D_{x}\phi$, integrate it over $(x,\infty)$, and use \eqref{sbc4} and $\lim_{x \to \infty}\D_{x}\phi(x)=0$ to obtain
\begin{equation}\lb{phieq6}
(\D_{x} \phi)^2=2V^{-}(\phi), \quad V^{-}(\phi)=\int_{0}^{\phi} \left(\rho_{i}^{-}(\varphi) - n_{e}(\varphi) \right)  d\vphi,
\end{equation}
where the function $V^{-}$ is the same function defined in \eqref{V-}.
Thus $\phi$ must satisfy \eqref{phieq6}.

Now we show $\rho_{i}^{-} \in C([\phi_{b},0])$ which immediately gives $V^{-} \in C^{1}([\phi_{b},0])$.
Owing to $ f \in C(\overline{\R_{+}};L^{1}(\R^{3}))$ and \eqref{weak2} in Definition \ref{DefS1}, there hold that
\begin{gather*}
\rho_{i}^{-}(\phi(x))=\| f(x) \|_{L^{1}} \in C(\overline{\R_{+}}), \quad
\lim_{x\to\infty}\| f(x) \|_{L^{1}} = \| f_{\infty} \|_{L^{1}}=1.
\end{gather*}
These imply $\rho_{i}^{-} \in C([\phi_{b},0])$ with the aid of \eqref{sbc4},
$ \phi \in C(\overline{\R_{+}})$, and $\D_{x} \phi >0$ in Definition \ref{DefS1}.

What is left is to obtain the generalized Bohm criterion \eqref{Bohm6}.
We see from $n_{e}(0)=1$ and \eqref{netrual1} that
\begin{equation}\lb{V6}
V^{-}(0)=\frac{d V^{-}}{d \phi}(0)=0.
\end{equation}
If $V^{-} \in C^{2}([\phi_{b},0])$, we arrive at \eqref{Bohm6} 
from \eqref{phieq6} and \eqref{V6} with the aid of the Taylor theorem. 
The proof is complete.
\end{proof}

Assertion (iii) in Theorem \ref{existence3} can be shown 
by the same method as in the proof of Assertion (iii) in Theorem \ref{existence1} with the aid of \eqref{phieq6}. 
We omit the proof.
Let us show Assertions (i) and (ii) in Theorem \ref{existence3}.

\begin{proof}[Proof of Assertions (i) and (ii) in Theorem \ref{existence3}]
In the same way as in the proof of Assertions (i) and (ii) in Theorem \ref{existence2}, 
we first see that $B^{-}\neq \emptyset$ if $d^{2} V^{-}/d\phi^{2}(0)>0$;
$\phi_{b} >\inf B^{-}$ is necessary;
 $\sqrt{V^{-}}$ is Lipschitz continuous on $[\phi_{b},0]$.
%
%
To find a solution $\phi$ with $\D_{x} \phi>0$ of the equation \eqref{phieq6} with \eqref{sbc3} and \eqref{sbc4},
we rewrite \eqref{phieq6} into the equivalent equation
\begin{gather*}
\D_{x}{\phi}=\sqrt{2V^{-}({\phi})}.
\end{gather*}
Then we complete the proof in the same manner as that of Assertions (i) and (ii) of Theorems \ref{existence1} and \ref{existence2}.
\end{proof}

Similarly as the proof of Corollary \ref{4.1},
we also see that  $\inf B^{-}$ in Assertion (i) of Theorem \ref{existence3} is bounded from above 
by some negative constant $-\delta_{0}$ independent of $\ve$
for the problem \eqref{sVP1} with $(f_{b},f_{\infty})=(g_{b,\ve},g_{\infty,\ve})$.
Here $g_{b,\ve}$ and $g_{\infty,\ve}$ defined in \eqref{fb2} and \eqref{fve2} satisfy
the conditions \eqref{netrual1}, \eqref{need4}, and $d^{2} V^{-}/d\phi^{2}(0)>0$ being
in Theorem \ref{existence3} for $|\phi_{b}| \ll 1$ and $\ve\ll 1$ provided that \eqref{velocity1} holds.
The result is summarized in the following corollary.

\begin{cor}\label{4.3}
Let \eqref{velocity1} hold.
There exist positive constants $\ve_{0}$ and $\delta_{0}$ such that 
if $-\delta_{0}<\phi_{b}<0$, then the problem \eqref{sVP1} with $(f_{b},f_{\infty})=(g_{b,\ve},g_{\infty,\ve})$
has a unique solution $(f_{\ve},\phi_{\ve}) \in C^{1}(\overline{\mathbb R_{+}}\times \mathbb R^{3}) \times C^{2}(\overline{\mathbb R_{+}})$ for any $\ve \in (0,\ve_{0})$.
\end{cor}

\subsection{The Delta Mass Limit}\label{S4.3}
This section provides the proof of Theorem \ref{DiffThm2} on the delta mass limit.
We start from showing the next lemma on the solvability of the problem \eqref{Hphi0}.

\begin{lem}
Let \eqref{velocity1} hold.
There exists a positive constant $\delta_{0}$ such that 
if $|\phi_{b}|<\delta_{0}$, then the problem \eqref{Hphi0} has a unique monotone solution $\phi \in C^{2}(\overline{\mathbb R_{+}})$.
\end{lem}
\begin{proof}
We first note that $\lim_{x \to \infty}\D_{x}\phi(x)=0$ must hold if such a solution $\phi$ exists.
Now multiply \eqref{Hphi1} by $\D_{x}\phi$ and integrate it over $(x,\infty)$ to obtain
\begin{gather*}
(\D_{x} \phi)^2=2\tilde{V}(\phi), \quad \tilde{V}(\phi):=\int_{0}^{\phi} \frac{m_{b}v_{b}}{\sqrt{v_{b}^{2}+2\vphi}}+
\frac{(1+\alpha)m_{\infty}v_{\infty}}{\sqrt{v_{\infty}^{2}+2\vphi}} - n_{e}(\varphi) \,  d\vphi.
\end{gather*}
It is straightforward to see that $\tilde{V} \in C^{2}([-\delta_{0},\delta_{0}])$ holds for some $\delta_{0}>0$.
From $n_{e}(0)=1$, $n_{e}'(0)=-1$, and \eqref{velocity1}, one can know that
\begin{gather*}
\tilde{V}(0)=\frac{d \tilde{V}}{d \phi}(0)=0, \quad \frac{d^{2} \tilde{V}}{d \phi^{2}}(0) >0.
\end{gather*}
Then following the proofs of Assertions (i) and (ii) of Theorems \ref{existence1}, \ref{existence2}, and \ref{existence3}, we can complete the proof.
\end{proof}

We are now in a position to prove Theorem \ref{DiffThm2}.

\begin{proof}[Proof of Theorem \ref{DiffThm2}]
Let $(f_{\ve},\phi_{\ve})$ be solutions in Corollaries \ref{4.2} and \ref{4.3}, 
and also $(\rho_{\ve},u_{\ve})$ be the moments of $f_{\ve}$.
We first show the estimate of $\phi_{\ve}-\phi$ in \eqref{conver5}. 
From \eqref{Hphi1}, \eqref{phieq5+}, and \eqref{phieq6+}, we observe that 
\begin{align*}
\partial_{xx} (\phi_{\ve}-\phi)
= W(\phi_{\ve}) - W(\phi) + R_{\ve}(x), \quad x>0,
\end{align*}
where we have used the fact that the second term vanishes in \eqref{rho+} for $f_{b}=g_{b,\ve}$ and $0<\phi_{b}\ll 1$,
and $W(\phi)$ and $R_{\ve}(x)$ are defined as
\begin{align*}
W(\phi)&:= \frac{m_{b}v_{b}}{\sqrt{v_{b}^{2}+2\phi}}+
\frac{(1+\alpha)m_{\infty}v_{\infty}}{\sqrt{v_{\infty}^{2}+2\phi}}
-n_{e}(\phi),
\\
R_{\ve}(x)&:=\int_{\mathbb R^{3}} g_{\infty,\ve}(\xi)\frac{|\xi_{1}|}{\sqrt{\xi_{1}^{2}+2\phi_{\ve}(x)}}  d\xi 
- \frac{m_{b}v_{b}}{\sqrt{v_{b}^{2}+2\phi_{\ve}(x)}}
-\frac{(1+\alpha)m_{\infty}v_{\infty}}{\sqrt{v_{\infty}^{2}+2\phi_{\ve}(x)}}.
\end{align*}
Set $M_{\ve}:=\sup_{x \in {\mathbb R}_{+}} \left| R_{\ve}(x)\right|$.
Owing to $n_{e}'(0)=-1$ and \eqref{velocity1}, we have ${d W }/{d\phi } (0)>0$.
Then it is seen by taking $\delta_{0}>0$ and $c_{0}>0$ small enough that 
\begin{gather*}
c_{0} \leq  \frac{d W }{d\phi } (\phi) \leq c_{0}^{-1}, \quad \phi \in [-\delta_{0},\delta_{0}].
\end{gather*}
Now following the proof of \eqref{DiffEq3}, we arrive at 
\begin{gather}\label{DiffEq13}
\sup_{x \in {\mathbb R}_{+}}|(\phi_{\ve}-\phi)(x)| \leq  M_{\ve}/c_{0}.
\end{gather}

For sufficiently small $\ve>0$, let us estimate $M_{\ve}$ by $C_{0} \ve$ as 
\begin{align*}
M_{\ve}
&= \sup_{x \in {\mathbb R}_{+}} \left|\int_{\mathbb R^{3}} g_{\infty,\ve}(\xi) \left(
\frac{|\xi_{1}|}{\sqrt{\xi_{1}^{2}+2\phi_{\ve}(x)}}
- \frac{m_{b}v_{b}}{\sqrt{v_{b}^{2}+2\phi_{\ve}(x)}}
-\frac{(1+\alpha)m_{\infty}v_{\infty}}{\sqrt{v_{\infty}^{2}+2\phi_{\ve}(x)}}  \right)d\xi \right| 
\notag \\
& \leq \sup_{x \in {\mathbb R}_{+}} \left(\sup_{|\xi-(-v_{\infty},0,0)| \leq \ve} 
\left|\frac{-\xi_{1}}{\sqrt{\xi_{1}^{2}+2\phi_{\ve}(x)}}-\frac{v_{\infty}}{\sqrt{v_{\infty}^{2}+2\phi_{\ve}(x)}} \right| \right)
m_{\infty}
\notag \\
&\quad +\sup_{x \in {\mathbb R}_{+}} \left(\sup_{|\xi-(v_{b},0,0)| \leq \ve} 
\left|\frac{\xi_{1}}{\sqrt{\xi_{1}^{2}+2\phi_{\ve}(x)}}-\frac{v_{b}}{\sqrt{v_{b}^{2}+2\phi_{\ve}(x)}} \right| \right)
m_{b}
\notag \\
&\quad +\sup_{x \in {\mathbb R}_{+}} \left(\sup_{|\xi-(v_{\infty},0,0)| \leq \ve} 
\left|\frac{\xi_{1}}{\sqrt{\xi_{1}^{2}+2\phi_{\ve}(x)}}-\frac{v_{\infty}}{\sqrt{v_{\infty}^{2}+2\phi_{\ve}(x)}} \right| \right)
\alpha m_{\infty}
\notag \\
& \leq C_{0} \ve,
\end{align*}
where we have used the fact 
$\int_{\mathbb R^{3}} g_{\infty,\ve}(\xi)d\xi=m_{b}+(1+\alpha)m_{\infty}=1$ in deriving the equality,
and $C_{0}$ is a positive constant independent of $\ve$.
This together with \eqref{DiffEq13} leads to the estimate of $\phi_{\ve}-\phi$ in \eqref{conver5}.

We complete the proof by showing the other estimates in \eqref{conver5}.
It is clear that $\rho_{\ve}(x)=\rho_{i}^{\pm}(\phi_{\ve}(x))$ holds owing to \eqref{rho+'} and \eqref{rho-'}.
We estimate $\rho_{\ve}-\rho$ as
\begin{align*}
&\sup_{x \in {\mathbb R}_{+}} \left| \rho_{\ve}(x)-\rho(x) \right|
\notag \\
&= \sup_{x \in {\mathbb R}_{+}} \left|\int_{\mathbb R^{3}} g_{\infty,\ve}(\xi) \left(
\frac{|\xi_{1}|}{\sqrt{\xi_{1}^{2}+2\phi_{\ve}(x)}}
- \frac{m_{b}v_{b}}{\sqrt{v_{b}^{2}+2\phi(x)}}
-\frac{(1+\alpha)m_{\infty}v_{\infty}}{\sqrt{v_{\infty}^{2}+2\phi(x)}}  \right)d\xi \right| 
\notag \\
& \leq \sup_{x \in {\mathbb R}_{+}} \left(\sup_{|\xi-(-v_{\infty},0,0)| \leq \ve} 
\left|\frac{-\xi_{1}}{\sqrt{\xi_{1}^{2}+2\phi_{\ve}(x)}}-\frac{v_{\infty}}{\sqrt{v_{\infty}^{2}+2\phi(x)}} \right| \right)
m_{\infty}
\notag \\
&\quad +\sup_{x \in {\mathbb R}_{+}} \left(\sup_{|\xi-(v_{b},0,0)| \leq \ve} 
\left|\frac{\xi_{1}}{\sqrt{\xi_{1}^{2}+2\phi_{\ve}(x)}}-\frac{v_{b}}{\sqrt{v_{b}^{2}+2\phi(x)}} \right| \right)
m_{b}
\notag \\
&\quad +\sup_{x \in {\mathbb R}_{+}} \left(\sup_{|\xi-(v_{\infty},0,0)| \leq \ve} 
\left|\frac{\xi_{1}}{\sqrt{\xi_{1}^{2}+2\phi_{\ve}(x)}}-\frac{v_{\infty}}{\sqrt{v_{\infty}^{2}+2\phi(x)}} \right| \right)
\alpha m_{\infty}
\notag \\
& \leq C_{0} (\ve+\sup_{x \in {\mathbb R}_{+}}|\phi(x)-\phi_{\ve}(x)|) \leq C_{0} \ve.
\end{align*}
Let us treat $\rho_{\ve}u_{\ve}-\rho u$. 
By the change of variables $\sqrt{\xi_{1}^{2}-2\phi(x)}=\pm\zeta_{1}$, 
the term $\rho_{\ve}u_{\ve}$ is rewritten as
\begin{gather*}
(\rho_{\ve}u_{\ve})(x)=\int_{\mathbb R^{3}} \xi_{1} g_{\infty,\ve}(\xi)  d\xi.
\end{gather*}
It is easy to see $\rho u=m_{b}v_{b}+(\alpha-1)m_{\infty}v_{\infty}$ from \eqref{Hrho}.
Then we arrive at
\begin{align*}
\sup_{x \in {\mathbb R}_{+}} \left| (\ro_{\ve}u_{\ve}) (x) - (\rho u)(x) \right|
&=\sup_{x \in {\mathbb R}_{+}} \left| \int_{\mathbb R^{3}} g_{\infty,\ve}(\xi) 
(\xi_{1}-m_{b}v_{b}-(\alpha-1)m_{\infty}v_{\infty}) d\xi\right| \leq  \ve.
\end{align*}
Consequently,  \eqref{conver5} holds.
\end{proof}

\medskip

\noindent
{\bf Acknowledge.} 
This work was supported by JSPS KAKENHI Grant Numbers 18K03364 and 21K03308.

\begin{appendix}

\section{Reduction}\label{A0}
In this section, we reduce the boundary value problem of 
\eqref{seq1}--\eqref{seq2} with conditions \eqref{sbc1},  \eqref{sbc2}, \eqref{sbc4}, and \eqref{sbc5} 
to the boundary value problem \eqref{sVP1}.
For simplicity, we treat the reduction only for the completely absorbing boundary, i.e. $f_{b}=\alpha=0$.

Suppose that the former boundary value problem has a solution $(f,\phi)$ with $\partial_{x} \phi(x)<0$.
It is sufficient to show that a value $\phi(0)$ is determined a priori by $f_{\infty}$.
Indeed we can construct the soluiton of the former problem by solving 
the problem \eqref{sVP1} with $\phi_{b}=\phi(0)$.
Let us find an a priori value $\phi(0)$.
Following the proof of Lemma \ref{lem31},
we see that $f$ must be written by the form \eqref{fform} even for the former problem.
Substituting \eqref{fform} into \eqref{sbc5} with $\alpha=0$ yields the \footnote{If there is no value $\phi(0)$ so that this condition holds, the boundary value problem of 
\eqref{seq1}--\eqref{seq2} with conditions \eqref{sbc1},  \eqref{sbc2}, \eqref{sbc4}, and \eqref{sbc5} admits no solution.}following condition:
\begin{gather*}
n_{e}(\phi(0)) v_{e}=\int_{\mathbb R^{3}} \xi_{1} f_{\infty}(-\sqrt{\xi_{1}^{2}-2\phi(0)},\xi')\chi(\xi_{1}^{2}-2\phi(0))\chi(-\xi_{1}) d\xi.
\end{gather*}
By solving this with respect to $\phi(0)$, we can know the a priori value $\phi(0)$. 
Consequently, the former problem can be reduced to 
the problem \eqref{sVP1} with $\phi_{b}=\phi(0)$.

\section{Estimates of $\rho^{\pm}$}\label{C}
This section provides the proofs of estimates \eqref{rho+} and \eqref{rho-}.
First we can obtain \eqref{rho+}  by using the H\"older inequality as follows: 
\begin{align*}
|\rho_{i}^{+}(\phi)| 
& \leq \|f_{\infty}\|_{L^{1}(\mathbb R^{3})} +C  \int_{\sqrt{2(\phi_b-\phi)\chi(\phi_b-\phi)}}^{\sqrt{2\phi_b}}  \frac{\xi_1}{\sqrt{\xi_1^2+2\phi-2\phi_b}} \left( \int_{\mathbb R^2} f_b(\xi)\, \,d\xi'  \right) d\xi_1
\\
&\leq \|f_{\infty}\|_{L^{1}(\mathbb R^{3})} +C \left( \int_{\sqrt{2(\phi_b-\phi)\chi(\phi_b-\phi)}}^{\sqrt{2\phi_b}}  \frac{|\xi_1|^{r'}}{({\xi_1^2+2\phi-2\phi_b})^{r'/2}} d\xi_1 \right)^{\!\!1/r'} \!\!\|f_{b}\|_{L^{r}(0,\sqrt{2\phi_{b}};L^{1}(\mathbb R^{2}))}
\\
& \leq \|f_{\infty}\|_{L^{1}(\mathbb R^{3})} + C\|f_{b}\|_{L^{r}(0,\sqrt{2\phi_{b}};L^{1}(\mathbb R^{2}))}
\end{align*}
for $\phi \geq 0$, where $r'<2$ is the H\"older conjugate of $r>2$. 
Similarly,  we observe that for $\phi \in [-M,0]$,
\begin{align*}
|\rho_{i}^{-}(\phi)| 
& = \int_{{\mathbb R}^3}f_\infty(\xi)\frac{|\xi_1|}{\sqrt{\xi_1^2+2\phi}} \chi(\xi_1^2-4M)\,d\xi 
\\
&\quad +\int_{{\mathbb R}^3}f_\infty(\xi)\frac{|\xi_1|}{\sqrt{\xi_1^2+2\phi}} \{\chi(\xi_1^2+2\phi)-\chi(\xi_1^2-4M)\}\,d\xi 
\\
& \leq  \sqrt{2}\|f_{\infty}\|_{L^{1}(\mathbb R^{3})} 
+ \int_{-2\sqrt{M}}^{2\sqrt{M}} \frac{|\xi_1|}{\sqrt{\xi_1^2+2\phi}} \chi(\xi_1^2+2\phi) \ \left( \int_{{\mathbb R}^2}f_\infty(\xi) d\xi' \right)d\xi_{1}
\\
& \leq \sqrt{2}\|f_{\infty}\|_{L^{1}(\mathbb R^{3})} +C_{M}\|f_{\infty}\|_{L^{r}(-2\sqrt{M},2\sqrt{M};L^{1}(\mathbb R^{2}))}.
\end{align*}
Thus \eqref{rho-} holds. The proofs are complete.

\section{Properties of $V^{\pm}$}\label{B0}

We investigate properties of the functions $V^{\pm}$ defined in \eqref{V+} and \eqref{V-}.

\begin{lem}{\rm (Attractive boundary)} \
Let $\phi_b>0$, $\alpha \neq 1$, and $f_b\in L^1(\R_+^3)$ satisfy $f_{b} \geq 0$ and
\begin{gather}\label{fb0}
f_b(\xi)=0, \quad 0<\xi_1<c_0
\end{gather}
for some  $c_0>0$.
Suppose that $f_{\infty} \in L^{1}(\mathbb R^{3})$ satisfies $f_{\infty} \geq 0$, \eqref{netrual1}, \eqref{need3}, and \eqref{Bohm2}.
Then there exists a positive constant $\delta$ such that if $0<\phi_{b}<\delta$,
the function $V^{+}$ belongs to $C^{2}([0,\phi_{b}])$ and satisfies ${d^{2} V^{+}}/{d \phi^{2}}(0)>0$.

\smallskip

\noindent
{\rm (Repulsive boundary)} \ Let $\phi_b<0$,  $f_b\in L^1(\R_+^3)$, and $f_{b} \geq 0$.
Suppose that $f_{\infty} \in L^{1}(\mathbb R^{3})$ satisfies $f_{\infty} \geq 0$, \eqref{netrual1}, \eqref{need4}, \eqref{Bohm2}, and
\begin{gather*}
f_\infty(\xi)=0, \quad - c_0<\xi_1<c_0
\label{fi}
\end{gather*}
for some $c_0>0$.
Then there exists a constant $\delta$ such that if $-\delta<\phi_{b}<0$, 
the function $V^{-}$ belongs to $C^{2}([\phi_{b},0])$ and satisfies ${d^{2} V^{-}}/{d \phi^{2}}(0)>0$.
\end{lem}

\begin{proof}
We show only the assertion for the attractive boundary, since the other one can be shown similarly.
Recalling the definition of $\rho_{i}^{+}$ in \eqref{rho++}, 
we see from \eqref{fb0} that the second term vanishes for sufficiently small $\phi_{b}$.
Therefore, we can complete the proof 
by following the same argument as in the last paragraph of the proof of Lemma \ref{lem31}, 
and also noting that
\begin{gather*}
\frac{d^{2} V^{+}}{d \phi^{2}}(0) 
=-\int_{\mathbb R^{3}}\xi_{1}^{-2}f_{\infty}(\xi) d\xi+1>0,
\end{gather*}
where we have used \eqref{Bohm2} in deriving the last inequality.
\end{proof}

\end{appendix}

\end{document}